\documentclass[a4paper,12pt]{article}

\usepackage{amsmath, wrapfig}
\usepackage{dsfont, a4wide, amsthm, amssymb, amsfonts, graphicx}
\usepackage{fancyhdr, xspace, psfrag, setspace, supertabular, color, enumerate}
\usepackage{hyperref}

\newtheorem{theorem}{Theorem}[section]
\newtheorem{proposition}[theorem]{Proposition}
\newtheorem{lemma}[theorem]{Lemma}

\newtheorem{construction}{Construction}

\newtheorem{conjecture}[theorem]{Conjecture}
\newtheorem{definition}[theorem]{Definition}

\newtheorem{example}[theorem]{Example}
\newtheorem{question}[theorem]{Question}

\def\a{\alpha}
\def\b{\beta}

\def\t{\theta}
\def\l{\lambda}

\def\cb{\mathcal{B}}
\def\ca{\mathcal{A}}

\def\ct{\mathcal{T}}

\def\F{\mathbb{F}}

\def\R{\mathbb{R}}

\DeclareMathOperator{\rk}{rk}
\DeclareMathOperator{\tr}{tr}
\DeclareMathOperator{\sr}{sr}
\DeclareMathOperator{\pr}{pr}

\title{Small sunflowers and the structure of slice rank decompositions}
\author{Thomas Karam\footnote{Mathematical Institute, University of Oxford. Email: \texttt{thomas.karam@maths.ox.ac.uk}.}}

\begin{document}
\maketitle

\begin{abstract}

Let $d \ge 3$ be an integer. We show that whenever an order-$d$ tensor admits $d+1$ decompositions according to Tao’s slice rank, if the linear subspaces spanned by their one-variable functions constitute a sunflower for each choice of special coordinate, then the tensor admits a decomposition where these linear subspaces are contained in the centers of these respective sunflowers. As an application, we deduce that for every nonnegative integer $k$ and every finite field $\F$ there exists an integer $C(d,k,|\F|)$ such that every order-$d$ tensor with slice rank $k$ over $\F$ admits at most $C(d,k,|\F|)$ decompositions with length $k$, up to a class of transformations that can be easily described.

\end{abstract}

\tableofcontents

\section{Introduction}

Throughout we will use the following notations. If $k$ is a positive integer, then $[k]$ will denote the set $\{1, \dots, k\}$ of positive integers up to $k$. The letter $\F$ will denote a field. All our statements will be uniform with respect to the choice of the field $\F$, unless we explicitly write the assumption that the field $\F$ is finite. Even in that latter case, the dependence will only be on the size of the field $\F$ and not involve its structure.

We will prove various statements involving order-$d$ tensors, with $d \ge 2$ some integer. Order-$d$ tensors will be taken to be functions $[n_1] \times \dots \times [n_d] \to \F$ for some positive integers $n_1, \dots, n_d$. All our statements will be uniform in the integers $n_1, \dots, n_d$ once all other parameters are fixed, and we will use the integers $n_1, \dots, n_d$ without defining them again in our statements.

Let $k$ be a positive integer, and let $\F$ be a field. If $M:[n_1] \times [n_2] \to \F$ is a matrix with rank $k$, then all decompositions of $M$ as a sum of $k$ rank-$1$ matrices can be obtained from one another up to changes of bases. In particular, there exist some linear subspaces $A_1 \subset \F^{n_1}, A_2 \subset \F^{n_2}$ such that if \[M(x,y) = \sum_{i=1}^k f_i(x) g_i(y)\] is a decomposition of $M$, then $\langle f_1, \dots, f_k \rangle = A_1$ and $\langle g_1, \dots, g_k\rangle = A_2$. If $\F$ is a finite field, then it is well-known that the number of decompositions of $M$, where two decompositions are viewed as the same if their $2k$-tuples $(f_1, \dots, f_k,g_1, \dots, g_k)$ are the same, is equal to the number \begin{equation} (|\F|^k - 1) (|\F|^k - |\F|) \dots (|\F|^k - |\F|^{k-1}) = |\F|^{k^2} \prod_{i=1}^{k} (1-|\F|^{-i}) \label{number of decompositions of matrices} \end{equation} of bases of $\langle f_1, \dots, f_k \rangle$. This number is between $\omega |\F|^{k^2}$ and $|\F|^{k^2}$, where $\omega$ is taken to be the absolute constant \[\prod_{i=1}^{\infty} (1-2^{-i}) >0.\]

We note that even from a purely qualitative perspective, the analogous boundedness statement becomes false in general for decompositions of length even one greater than the rank of $M$. For instance, if $M(x,y) = f(x) g(y)$, then for any functions $f_1,f_2: [n] \to \F$ satisfying $f_2 - f_1 = f$, we have \[M(x,y) = f_2(x) g(y) - f_1(x) g(y),\] so as long as the function $f_2$ is not fixed, the function $f_1$ can be completely arbitrary.

Our main aim in the present paper is to obtain a comparable statement for the notion of slice rank on higher-order tensors. We have found the task of formulating an appropriate corresponding statement to already be challenging, which is why we shall state it precisely only in Section \ref{Section: Statements of main results}, after going through several constructions guiding us towards its formulation and ruling out stronger ones in Section \ref{Section: Constructions of different minimal length slice rank decompositions of the same tensor}. Informally speaking, we will prove that up to a natural class of transformations, the number of minimal-length slice rank decompositions of a tensor over a finite field is bounded above in a way that depends only on the order of the tensor, on its slice rank, and on the size of the field. As we will also discuss, the bound that we obtain cannot be too far from the optimal bound.  Let us recall the definition of the slice rank and even before this, the definition of the tensor rank.

\begin{definition} Let $d \ge 2$ be an integer, and let $T: [n_1] \times \dots \times [n_d] \to \F$ be an order-$d$ tensor. We say that the \emph{tensor rank} of $T$, denoted by $\tr T$, is the smallest nonnegative integer $k$ such that there exist functions $a_{j,i}: [n_j] \to \F$ for each $j \in [d]$ and $i \in [k]$ satisfying

\begin{equation} T(x_1, \dots, x_d) = \sum_{i=1}^k a_{1,i}(x_1) a_{2,i}(x_2) \dots a_{d,i}(x_d) \label{tensor rank decomposition} \end{equation} is satisfied for all $(x_1, \dots, x_d) \in [n_1] \times \dots \times [n_d]$.

We say that an expression such as \eqref{tensor rank decomposition} is a \emph{tensor rank decomposition} of $T$, say that the \emph{length} of the decomposition is the integer $k$, and say that the decomposition has \emph{minimal length} if its length is equal to the tensor rank of $T$. \end{definition}

For every $x \in [n_1] \times \dots \times [n_d]$ and every $j \in [d]$, we write $\overline{x_j}$ for $(x_{j’})_{j’ \in [d] \setminus \{j\}}$.

\begin{definition} Let $d \ge 2$ be an integer, and let $T: [n_1] \times \dots \times [n_d] \to \F$ be an order-$d$ tensor. We say that the \emph{slice rank} of $T$, denoted by $\sr T$, is the smallest nonnegative integer $k$ such that there exist nonnegative integers $r_1, \dots, r_d$ with $r_1 + \dots + r_d = k$ satisfying one of the following two equivalent properties. \begin{enumerate}

\item There exist matrices $M_j: [n_j] \times (\prod_{j' \neq j} [n_{j'}]) \to \F$ with rank at most $r_j$ for each $j \in [d]$ such that \[T(x_1, \dots, x_d) = \sum_{j=1}^d M_j(x_j, \overline{x_j})\] is satisfied for all $(x_1, \dots, x_d) \in [n_1] \times \dots \times [n_d]$.

\item For each $j \in [d]$ and $i \in [r_j]$ there exist some functions $a_{j,i}: [n_j] \to \F$ and $b_{j,i}: \prod_{j' \neq j} [n_{j'}] \to \F$ such that \begin{equation} T(x_1, \dots, x_d) = \sum_{j=1}^d \sum_{i=1}^{r_j} a_{j,i}(x_j) b_{j,i}(\overline{x_j}) \label{slice rank decomposition} \end{equation} is satisfied for all $(x_1, \dots, x_d) \in [n_1] \times \dots \times [n_d]$. \end {enumerate} 

We say that an expression such as \eqref{slice rank decomposition} is a \emph{slice rank decomposition} of $T$, say that the \emph{length} of the decomposition is the integer $r_1 + \dots + r_d$, and say that the decomposition has \emph{minimal length} if its length is equal to the slice rank of $T$.

\end{definition}

Before going further, let us recall the history of the slice rank and some of the ways in which it has been studied.

The slice rank was originally introduced by Tao \cite{Tao} in 2016 as a reformulation of a central idea from the breakthrough of Croot, Lev and Pach \cite{Croot Lev Pach}, which led to the solution to the cap-set problem by Ellenberg and Gijswijt \cite{Ellenberg Gijswijt}. Since then, the slice rank has been used successfully several times as a tool to solve combinatorial problems. For instance, it was also shown by Naslund and Sawin \cite{Naslund and Sawin} that subsets of the cube $\{0,1\}^n$ containing no $3$-sunflower are exponentially sparse, and further properties on the slice rank from Sawin and Tao \cite{Sawin and Tao} were later used by Sauermann \cite{Sauermann} to obtain properties guaranteeing the existence of solutions with pairwise distinct variables to systems of equations in subsets of vector spaces over finite prime fields which are not exponentially sparse. In another direction, a variant of the slice rank, the partition rank, was defined by Naslund in \cite{Naslund}, originally to obtain polynomial upper bounds on problems on $k$-right corners, and was more recently used again by Naslund \cite{Naslund recent}, this time to obtain exponential lower bounds on the chromatic number of $\R^n$ with multiple forbidden distances.

Although all these works have primarily used the slice rank and the partition rank as tools, interest in studying them and their basic properties for their own sake has recently been building up. In the post of Sawin and Tao \cite{Sawin and Tao}, a characterisation of the slice rank in terms of coverings in the case of tensors supported inside an antichain had already been discussed, and this characterisation later played an important role in the work of Sauermann \cite{Sauermann} that we previously mentioned. Later, it was proved by Gowers \cite{Gowers} that the slice rank of a direct sum of two tensors is the sum of their slice ranks.

As another example in this direction, the facts that a high-rank matrix must contain a high-rank submatrix of not too large size and that a matrix with high rank after every modification of the diagonal must have a high rank submatrix for which the sets of rows and columns are disjoint were extended by the author in \cite{K.} to a class of notions of rank containing the slice and partition ranks. Still on the topic of subtensors, it was shown by Bri\"et and Castro-Silva \cite{Briet Castro-Silva}, again for a wide class of notions of rank, that a random subtensor of a high-rank subtensor must have high rank for some natural way of choosing the restriction at random, together with some analogous results on random restrictions of polynomials.

These three results found applications. The first, conditionally on better bounds, was discussed in \cite{Briet Castro-Silva} as a way to provide an alternate proof of some of the main results of that paper. The second was used by Gowers and the author \cite{Gowers and K equidistribution} as a key stepping stone to extend to distributions inside $\{0,1\}^n$ the result of Green and Tao \cite{Green and Tao} on the equidistribution of high-rank polynomials over finite prime fields. The third was applied (primarily in the case of polynomials, although the case of tensors was involved as well) by Bri\"et, Buhrman, Castro-Silva, and Neumann \cite{Briet Buhrman Castro-Silva Neumann} to prove limitations on the decoding of corrupted classical and quantum error-correcting codes with $\mathrm{NC}^0[\oplus]$ circuits. 

The slice and partition ranks have also been studied using more algebraic methods. For instance, a “universal" role that the partition rank plays among notions of rank has been explored by Bik, Draisma and Eggermont \cite{Bik Draisma Eggermont}. Another universality result was proved by Kazhdan and Ziegler \cite{Kazhdan and Ziegler} on the strength of polynomials (an analogue of the partition rank of tensors), and later led to an extension by Bik, Danelon, Draisma, and Eggermont \cite{Bik Danelon Draisma Eggermont} to the more abstract setting of arbitrary polynomial functors.

Let us finally mention that in the case of the tensor rank, much effort has gone into understanding the notion of identifiability of a tensor, which is much stronger than the mere boundedness of its set of minimal-length decompositions that we focus on (primarily in the case of the slice rank) in the present paper: if $k$ is a nonnegative integer, then a tensor $T$ with tensor rank $k$ is said to be identifiable if there is only one way of writing it as a sum \[T_1 + \dots + T_k,\] up to permutations of the $k$ tensors $T_1, \dots, T_k$. This area of research, perhaps started by Kruskal \cite{Kruskal} in 1977, is still very active. A recent paper \cite{Ballico Bernardi Santarsiero} of Ballico, Bernardi and Santarsiero contains, among other interesting things, a wealth of references on identifiability.

We now define some more notations which we shall use throughout the paper. Let $d \ge 2$ be an integer. If $x$ is an element of $[n_1] \times \dots \times [n_d]$ and $J$ is a subset of $[d]$, then we shall write $x(J)$ for the restriction of $x$ to its coordinates in $J$, that is, for the element of $\prod_{j \in J} [n_j]$ defined by $x(J)_j = x_j$ for every $j \in J$. If $a,a':[n_j] \to \F$ are functions for some $j \in [d]$, then we write $a.a'$ for the element of $\F$ defined by \[\sum_{x_j \in [n_j]} a(x_j) a'(x_j).\] More generally, if $J' \subset J$ are non-empty subsets of $[d]$, and $U: \prod_{j \in J} [n_j] \to \F$, $U': \prod_{j \in J'} [n_j] \to \F$ are functions, then $U'.U: \prod_{j \in J \setminus J'} [n_{j}] \to \F$ will denote the function (or rather element of $\F$ if $J'=J$) defined by \[(U'.U)(x(J \setminus J')) = \sum_{x(J') \in \prod_{j \in J'} [n_j]} U'(x(J')) U(x(J))\] for all $x(J \setminus J')$. If $s$ is a positive integer, and $j_1, \dots, j_s$ are pairwise distinct elements of $[d]$, then we write $\overline{x_{j_1}, \dots, x_{j_s}}$ for $x([d] \setminus \{j_1, \dots, j_s\})$.

We will repeatedly use the following observations. If a decomposition \eqref{slice rank decomposition} has minimal length, then for every $j \in [d]$ its functions $a_{j,i}$ with $i \in [r_j]$ are linearly independent. Even if the decomposition does not have minimal length, each of the $d$ parts of the decomposition and hence the whole decomposition can always be written such that these functions are linearly independent, and we will always assume this to be the case whenever we write any slice rank decomposition of any tensor.

It follows from Gaussian elimination that if for some $j \in [d]$ and some integer $r \ge 1$ the functions $a_1, \dots, a_r: [n_j] \to \F$ are linearly independent, then there exist functions $a_1^*, \dots, a_r^*: [n_j] \to \F$ satisfying $a_i^*.a_{i'} = 1_{i=i'}$ for any $i,i' \in [r]$. We will refer to the family $(a_1^*, \dots, a_r^*)$ as a family of \emph{dual functions} to the family of functions $(a_1, \dots, a_r)$. More generally, if for some non-empty subset $J$ of $[d]$ and some integers $r_j \ge 1$ for each $j \in J$ the functions $a_{j,1}, \dots, a_{j,r_j}$ are linearly independent for every $j \in J$, then the tensor products \[\otimes_{j \in J} a_{j,i_j}: \prod_{j \in J} [n_j] \to \F \] with $(i_j)_{j \in J}$ are linearly independent, and furthermore \[(\otimes_{j \in J} a_{j,i_j}^*). (\otimes_{j \in J} a_{j,i_j'}) = \prod_{j \in J} 1_{i_j = i_j’}\] is equal to $1$ if $i_j=i_j’$ for every $j \in J$, and to $0$ otherwise.

The remainder of the paper is organised as follows. In Section \ref{Section: Constructions of different minimal length slice rank decompositions of the same tensor} we describe several rather systematic ways in which we may construct several slice rank decompositions of the same tensor. Against these constructions we then formulate in Section \ref{Section: Statements of main results} our result on the structure of slice rank decompositions, Theorem \ref{Main theorem}, together with two other results on slice rank decompositions, Proposition \ref{order-$d$ decompositions of zero} and Theorem \ref{At most d decompositions}, from which Theorem \ref{Main theorem} can be deduced and which also appear to be of intrinsic interest. This deduction is then performed in Section \ref{Section: Deduction of the structure theorem on slice rank decompositions}. In Section \ref{Section: the tensor rank case}, which is independent from the other sections, we give a simple proof of a statement showing that in the case of the tensor rank, the result on decompositions of matrices extends in a rather optimistic way. In Section \ref{Section: The proof for order-$3$ tensors} we then prove Proposition \ref{order-$d$ decompositions of zero} and Theorem \ref{At most d decompositions} in the case of order-$3$ tensors, and in Section \ref{Section: The proof in the general case} we prove them in the general case of order-$d$ tensors for every $d \ge 3$. Finally, in Section \ref{Section: Open questions} we conclude by mentioning some questions which remain open.

Most of the underlying basic ideas involved in our proofs of the general $d$ case already arise for $d=3$, and these are discussed in more detail in the $d=3$ case, which is why Section \ref{Section: The proof for order-$3$ tensors} is slightly lengthier than Section \ref{Section: The proof in the general case} despite proving a less general result.

\section*{Acknowledgements}

The author thanks Timothy Gowers for discussions around Gowers's paper \cite{Gowers} that attracted the author's attention to the idea of slice rank decompositions of the zero tensor, which was involved there and plays a role in the present paper as well. The author also thanks Jordan Ellenberg for encouraging him to prove results stating that even if the analogue of a “perfect" property which holds for the matrix rank does not hold for other notions of rank on tensors, some appropriate weakening of it does.

\section{Constructions of different minimal-length slice rank decompositions} \label{Section: Constructions of different minimal length slice rank decompositions of the same tensor}

We next go over five ways in which several minimal-length slice rank decompositions can arise for the same tensor.

\begin{construction} \label{Construction 1} \emph{Since it is possible to obtain new decompositions of matrices using a change of basis, it is possible to obtain new decompositions of $T$ from an existing decomposition by rewriting the decompositions of any of the individual matrices $M_1, \dots, M_d$ while leaving the $d$-tuple $(M_1, \dots, M_d)$ unchanged.} \end{construction}

However, this $d$-tuple is not unique in general, and there are other ways of obtaining new decompositions, which brings us to the remaining four constructions.

\begin{construction} \label{Construction 2} \emph{To see that the matrices $M_1, \dots, M_d$ are not individually determined, consider any tensor $T$ for which the slice rank and tensor rank are equal to some common positive integer $k$. Let \[T(x_1, \dots, x_d) = \sum_{i=1}^k a_{1,1}(x_1) \dots a_{d,1}(x_d)\] be a tensor rank decomposition of $T$ with length $k$. Then for every partition $\{I_1, \dots, I_d\}$ of $[k]$ into $d$ sets $I_1, \dots, I_d$ that are possibly empty, the decomposition \[T(x_1, \dots, x_d) = \sum_{j=1}^d \sum_{i \in I_j} a_{j,i}(x_i) (\prod_{j’ \neq j} a_{j’,i}(x_{j’})) \] is a slice-rank decomposition of $T$ with length $k$.}

\emph{These tensors exist, since for instance the “identity" tensor $I_{d,k}: [k]^d \to \F$ defined by \[I_{d,k}(x_1, \dots, x_d) = \sum_{i=1}^k 1_{x_1=i} \dots 1_{x_d=i}\] was shown by Tao (\cite{Tao}, Lemma 1) to have slice rank $k$. Since the tensor rank of this tensor is at most $k$ by definition, and since the slice rank is always at most the tensor rank, it follows that both ranks are equal for this tensor.}

\emph{For instance, in the case $d=3$, we obtain a slice rank decomposition \[I_{3,k}(x,y,z) = \sum_{i \in I_1} 1_{x=i} 1_{y=z=i} + \sum_{i \in I_2} 1_{y=i} 1_{x=z=i} + \sum_{i \in I_3} 1_{z=i} 1_{x=y=i}\] for every tripartition $\{I_1,I_2,I_3\}$ of $[k]$ (again, with $I_1,I_2,I_3$ allowed to be empty).}

\end{construction}

As we can see from Construction \ref{Construction 2} not even the sizes of the three parts of the decomposition are determined in general. One comment that can nonetheless be made about Construction \ref{Construction 2} is that although the matrices $M_1, \dots, M_d$ are not always the same, the set of all summands is always the same set \[\{1_{x_1=i} \dots 1_{x_d=i}: i \in [k]\}.\] However, there is another class of examples where this usually does not hold, even up to changes of bases in the matrices $M_1, \dots, M_d$, and which even encompasses typical tensors.

\begin{construction} \label{Construction 3} \emph{A tensor $T:[k]^d \to \F$ can in particular be written in $d$ different ways as a sum of its $(d-1)$-variable slices of each type, i.e. be decomposed as \[T(x_1, \dots, x_d) = \sum_{i=1}^k 1_{x_j=i} T(x_1, \dots, x_{j-1}, i, x_{j+1}, \dots, x_d).\] If $T$ has slice rank $k$, then these $d$ decompositions are each minimal-length decompositions. If the field $\F$ is finite, then the number of tensors $[k]^d \to \F$ with slice rank at most $k-1$ is at most \[k^d |\F|^{(k-1)(k^{d-1}+k)} = k^d |\F|^{k^d-k^{d-1}+k^2-k}.\] Provided that $d \ge 3$, we can crudely bound this number above by \[(k^d / |\F|^{k}) |\F|^{k^d}.\] The ratio is less than $1$ for $k$ large enough, and tends to $0$ as $k$ tends to infinity, which shows that random tensors have slice rank $k$. Unlike in the $I_{d,k}$ example, the sets of $k$ summands involved in each of the $d$ decompositions are in particular usually very different.} \end{construction}

\begin{construction} \label{Construction 4} \emph{Provided that a tensor $T$ with slice rank $k$ does have a decomposition with length $k$ that does not have all its terms of the same of the three types, there is a systematic way of changing some of the components of the decomposition and obtain another decomposition with length $k$. In the $d=3$ case, for instance, if \[T(x,y,z) = a(x) b(y,z) + c(y) d(x,z),\] then for any function $e: [n_3] \to \F$, the functions $b' = b + c \otimes e$ and $d' = d - a \otimes e$ satisfy \[T(x,y,z) = a(x) b'(y,z) + c(y) d'(x,z).\]}

\emph{Because the function $e$ is arbitrary, and in particular there are unboundedly many such functions as $n$ tends to infinity, it is never true, for any integers $d \ge 3$, $k \ge 2$ and any field $\F$, that every order-$d$ tensor with slice rank $k$ has a number of minimal-length decompositions which is bounded above by some function of $d,k,\F$. A boundedness statement of this kind may only hold up to this type of transformation.}

\emph{If $T$ is an order-$3$ tensor with a more general slice rank decomposition \begin{equation} T(x,y,z) = \sum_{i=1}^r a_i(x) b_i(y,z) + \sum_{j=1}^s c_j(x) d_j(y,z) + \sum_{k=1}^t e_k(z) f_k(x,y) \label{order-3 slice rank decomposition} \end{equation} is a decomposition of $T$, then whenever $(i,j) \in [r] \times [s]$ and $e: [n_3] \to \F$ is a function, replacing $b_i$ and $d_j$ by respectively \[b_i + c_j \otimes e \text{ and } d_j - a_i \otimes e\] leads to a new decomposition of $T$.}

\emph{Returning to the case of general $d \ge 3$, if \[T(x_1,\dots,x_d) = a_1(x_1) b_1(\overline{x_1}) + \dots + a_d(x_d) b_d(\overline{x_d})\] is a slice rank decomposition of $T$, then for any function $c: [n_3] \times \dots \times [n_d]$, replacing $b_1$ and $b_2$ by respectively \[b_1 + a_2 \otimes c \text{ and } b_2 - a_1 \otimes c\] leads to a new decomposition of $T$.}

\emph{Combining both axes of generality, if \[T(x_1, \dots, x_d) = \sum_{j=1}^d \sum_{i=1}^{r_j} a_{j,i}(x_j) b_{j,i}(\overline{x_j})\] is a slice rank decomposition of $T$, $c: [n_3] \times \dots \times [n_d]$ is a function, and $i_1 \in [r_1]$, $i_2 \in [r_2]$ are two indices, then replacing $b_{1,i_1}$ and $b_{2,i_2}$ by respectively \[b_{1,i_1} + a_{2,i_2} \otimes c \text{ and } b_{2,i_2} - a_{1,i_1} \otimes c\] leads to a new decomposition of $T$. The roles of either of the first two coordinates may of course be exchanged with those of the other coordinates.} \end{construction}

\begin{construction} \label{Construction 5} \emph{We can generalise Construction \ref{Construction 4} further. Let us begin by describing a variant of Construction \ref{Construction 4}, in the simplest setting. If we have the decomposition \[T(x,y,z) = a(x) b(y,z) + c(y) d(x,z) + e(z) f(x,y),\] then the functions \begin{align*} b’ &= b + \lambda_1 c \otimes e\\
d’ &= d + \lambda_2 a \otimes e\\
f’ &= f + \lambda_3 a \otimes c \end{align*} for some $\lambda_1, \lambda_2, \lambda_3 \in \F$ adding up to $0$ satisfy \[T(x,y,z) = a(x) b’(y,z) + c(y) d’(x,z) + e(z) f’(x,y).\]}

\emph{We note that this transformation reduces to successively applying three of the transformations from Construction \ref{Construction 4}, since we can always find $\mu_{12,1}$, $\mu_{12,2}$, $\mu_{13,1}$, $\mu_{13,3}$, $\mu_{23,2}$, $\mu_{23,3} \in \F$ satisfying \begin{align*} \mu_{12,1} + \mu_{12,2} = 0 &\text{ \space \space \space \space \space \space \space \space \space \space} \lambda_1 = \mu_{12,1} + \mu_{13,1} \\
\mu_{13,1} + \mu_{13,3} = 0 &\text{ \space \space \space \space \space \space \space \space \space \space} \lambda_2 = \mu_{12,2} + \mu_{23,2} \\
\mu_{23,2} + \mu_{23,3} = 0 &\text{ \space \space \space \space \space \space \space \space \space \space} \lambda_3 = \mu_{13,3} + \mu_{23,3}.\end{align*} We may for instance choose $\mu_{13,1}$ and $\mu_{13,3}$ to be zero, and $\mu_{12,1}, \mu_{12,2}, \mu_{23,2}, \mu_{23,3}$ are then uniquely determined.}

\emph{If \eqref{order-3 slice rank decomposition} is a more general decomposition of an order-$3$ tensor $T$, then whenever $(i,j,k) \in [r] \times [s]\times [t]$ and $\lambda_1, \lambda_2, \lambda_3 \in \F$ satisfy \[\lambda_1 + \lambda_2 + \lambda_3 = 0,\] replacing $b_i, d_j, f_k$ by respectively \begin{align*} b_i & + \lambda_1 c_j \otimes e_k \\ d_j & + \lambda_2 a_i \otimes e_k \\ f_k & + \lambda_3 a_i \otimes c_j \end{align*} leads to a new decomposition of $T$.}

\emph{In the case of general $d \ge 3$, if \eqref{slice rank decomposition} is a decomposition of $T$, then we can choose some index $i_j$ for each $j \in [d]$, and choose some $\lambda_1, \dots \lambda_d \in \F$ satisfying \[\lambda_1 + \dots + \lambda_d = 0.\] Replacing the functions $b_{j,i_j}$ for each $j \in [d]$ by the functions \[b_{j,i_j}'(\overline{x_j}) = b_{j,i_j}(\overline{x_j}) + \lambda_j a_{1,i_1}(x_1) \dots a_{j-1,i_{j-1}}(x_{j-1}) a_{j+1,i_{j+1}}(x_{j+1}) \dots a_{d,i_d}(x_d)\] then provides a new decomposition of $T$.}

\emph{However, we can extend this class of transformations yet further: we can choose successively a subset $J$ of $[d]$ with size at least $2$, some indices $i_j \in [r_j]$ for each $j \in J$, some functions $c_j: \prod_{j \in [d] \setminus J} [n_j] \to \F$ for each $j \in J$ (instead taken to be a elements of $\F$ if $J = [d]$) satisfying \[\sum_{j \in J} c_j = 0,\] and replace for each $j \in J$ the function $b_{j,i_j}$ by the function $b_{j,i_j}'$ defined by \[b_{j,i_j}'(\overline{x_j}) = b_{j,i_j}(\overline{x_j}) + (\prod_{j' \in J \setminus \{j\}} a_{j',i_{j'}}(x_{j'})) c_j(x([d] \setminus J)). \] This yet again provides a new decomposition of $T$.}

\emph{Just as in the case of order-$3$ tensors, for general $d \ge 3$ and $2 \le |J| \le d$ doing this transformation ultimately reduces to successively applying transformations from Construction \ref{Construction 4} several times. For instance, writing $J = \{j_1, \dots, j_s\}$, there is a unique way of performing them while only using the $s-1$ pairs of indices $\{j_1, j_2\}, \dots, \{j_{s-1}, j_s\}$. Nonetheless, we single out Construction \ref{Construction 5} separately, as in our proofs we will think of the corresponding transformation as a single transformation rather than as a succession of several transformations from Construction \ref{Construction 4}.}

\end{construction}

Any of the five constructions that we have just described may be combined together: as shown by Gowers in \cite{Gowers}, the slice rank of a block diagonal tensor is equal to the sum of the slice ranks of the diagonal blocks, so putting together minimal-length decompositions of the five examples each respectively corresponding to a tensor $T_1, \dots, T_5$ leads to a minimal-length decomposition of the diagonal sum \[T_1 \oplus \dots \oplus T_5,\] and this is still the case after modifying any of the decompositions of $T_1, \dots, T_5$ using Constructions \ref{Construction 1} to \ref{Construction 5} respectively.

\section{Statements of main results} \label{Section: Statements of main results}

The previous section was devoted to constructions illustrating how a variety of different slice rank decompositions may arise for the same tensor. In the present section we describe our main results, which will be in the converse direction and show that to some extent, the set of minimal-length slice rank decompositions of the same tensor cannot be too rich. We will do so separately for the $(d-1)$-variable functions and for the one-variable functions of minimal-length decompositions. We will go through the following three steps, corresponding respectively to the proofs of the three main results of this paper, Proposition \ref{order-$d$ decompositions of zero}, Theorem \ref{At most d decompositions}, and Theorem \ref{Main theorem}.

\begin{enumerate}

\item We will begin by showing that if two decompositions have the same one-variable functions, then any differences between the two decompositions arise from Construction \ref{Construction 4} and Construction \ref{Construction 5}. 

\item After this, we will show a statement which in particular implies that no more than $d$ decompositions can have their sets of one-variable functions in the $j$th coordinate be jointly linearly independent for every $j \in [d]$ simultaneously.

\item This result will allow us to deduce that, in the finite field case, there are only a bounded number of possibilities for the one-variable functions, with a bound depending only on the order of the tensor, on its slice rank, and on the size of the finite field.

\end{enumerate}

In summary, we can partition the decompositions according to their one-variable functions in a bounded number of sets, and inside each of these sets we can then completely describe the possibilities for the $(d-1)$-variable functions once we know what these are for one of the decompositions from the set.

It may seem slightly surprising that we begin with and use fewer steps to obtain a satisfactory result with the $(d-1)$-variable functions than we do with the one-variable functions. That this is the case is due to the fact that whenever we compare two decompositions in the part of our argument studying the former, we always assume that the one-variable functions are the same, whereas while studying the latter we do not make any assumption on the $(d-1)$-variable functions.

Let us now go through these steps more formally. Although the following definition is a bit long, all that it describes is differences between the original and the new $(d-1)$-variable functions before and after successively applying Construction \ref{Construction 4} and Construction \ref{Construction 5} as much as can be done. If $J$ is a subset of $[d]$, then for simplicity of notation we will denote a sequence $(i_j)_{j \in J}$ by $i_J$.

\begin{definition} \label{zero form} Let $d \ge 3$ be an integer, and let $r_1, \dots, r_d$ be nonnegative integers. We say that a decomposition \[\sum_{j=1}^d \sum_{i=1}^{r_j} a_i(x_j) b_i(\overline{x_j})\] is in \emph{zero form} if there exist functions \[c_{J,j,i,i_{J \setminus \{j\}}}: \prod_{j'' \in [d] \setminus J} [n_{j''}] \to \F \] (instead taken to be elements of $\F$ if $J = [d]$) for every subset $J \subset [d]$ with $2 \le |J| \le d$, every $j \in [d]$, every $i \in [r_j]$ and every $i_{J \setminus \{j\}}$ satisfying the following two properties. \begin{enumerate}
\item For every $j \in [d]$ and every $i \in [r_j]$ we can write \[ b_{j,i}(\overline{x_j}) = \sum_{J \subset [d]: j \in J} \sum_{i_{J \setminus \{j\}}} \left( \prod_{j' \in J \setminus \{j\}} a_{j', i_{j'}}(x_{j'}) \right) c_{J,j,i,i_{J \setminus \{j\}}}(x([d] \setminus J)). \] 
\item For every $J \subset [d]$ with $2 \le |J| \le d$ and every $(i_{j'})_{j' \in J}$, we obtain a sum of $0$ whenever we take the sum over all functions $c_{J,j,i,i_{J \setminus \{j\}}}$ where the sequence $i_{J \setminus \{j\}}$ completed to a sequence indexed by $J$ by introducing the additional term $i_j=i$ is equal to $(i_{j'})_{j' \in J}$. \end{enumerate} \end{definition}

Given this definition, the first component of our results can be easily formulated. It can be checked that any decomposition that is in zero form adds up to the zero tensor, and our first result is a converse of that. We assume as usual that the functions $a_{j,i}$ of decompositions are linearly independent for each $j \in [d]$.

\begin{proposition} \label{order-$d$ decompositions of zero} Let $d \ge 3$ be an integer, and let $r_1, \dots, r_d$ be nonnegative integers. If a decomposition \[\sum_{j=1}^d \sum_{i=1}^{r_j} a_{j,i}(x_j) b_{j,i}(\overline{x_j})\] is equal to the zero tensor, then it is in zero form. \end{proposition}

The next theorem corresponds to the second step of the argument that we described at the beginning of this section. The key structure that we will use is that of sunflowers in a linear algebra sense: if $A^0,A^1,\dots,A^h$ are linear subspaces which are all jointly in direct sum for some positive integer $h$, then we say that the linear subspaces \[A^0 \oplus A^1, \dots, A^0 \oplus A^h\] constitute a \emph{sunflower} with \emph{center} $A^0$ and \emph{petals} $A^1,\dots,A^h$. Given bases of the linear subspaces $A^0,A^1,\dots,A^h$, we will refer to a basis of $A^0$ as a \emph{basis of center functions} and to bases of $A^1,\dots,A^h$ as \emph{bases of petal functions}.

What we will show is that if we can find $d+1$ decompositions $\t$ of the same tensor (not necessarily with minimal length) where for each $j \in [d]$ the linear subspaces spanned by the one-variable functions $a_{j,i}$ coming from the respective decompositions constitute a sunflower in the sense that we have just described, then the tensor has a slice rank decomposition where for every $j \in [d]$ the one-variable functions $a_{j,i}$ all belong to the center of the sunflower, and hence in particular the tensor has slice rank at most \[\dim A_1^0 + \dots + \dim A_d^0.\] We state this result for any $d \ge 2$: the $d=2$ case is relevant since as we shall discuss later in Section \ref{Section: The proof for order-$3$ tensors}, it will be used in the proof of the $d=3$ case.

\begin{theorem} \label{At most d decompositions} Let $d \ge 2$ be an integer, let $\F$ be a field, and let $T$ be an order-$d$ tensor over $\F$. Let $r_1^0,\dots, r_d^0$ be nonnegative integers, and let $a_{j,i}^0: [n_j] \to \F$ be one-variable functions for every $j \in [d]$ and every $i \in [r_j^0]$. Let $h>d$ be an integer, and assume that the following two conditions are satisfied. \begin{enumerate} \item For every $\t \in [h]$ there exist nonnegative integers $r_1^\t, \dots, r_d^\t$ and a decomposition of $T$ of the type \[\sum_{j=1}^d \left( \sum_{i=1}^{r_j^0} a_{j,i}^0(x_j) b_{j,i}^{0,\t}(\overline{x_j}) + \sum_{i=1}^{r_j^\t} a_{j,i}^{\t}(x_j) b_{j,i}^{\t}(\overline{x_j}) \right)\] for some one-variable functions $a_{j,i}^{\t}, [n_{j}] \to \F$ and for some $(d-1)$-variable functions $b_{j,i}^{0,\t}, b_{j,i}^{\t}: \prod_{j' \neq j} [n_{j'}] \to \F$. \item For each $j \in [d]$ the one-variable functions \[a_{j,1}^0, \dots, a_{j,r_j^0}^0, a_{j,1}^1, \dots, a_{j,r_j^1}^1, \dots, a_{j,1}^h, \dots, a_{j,r_j^h}^h: [n_j] \to \F \] are linearly independent. \end{enumerate} Then there exist $(d-1)$-variable functions $b_{j,i}^0: \prod_{j' \neq j} [n_{j'}] \to \F$ such that \[ \sum_{j=1}^d \sum_{i=1}^{r_j^0} a_{j,i}^0(x_j) b_{j,i}^0(\overline{x_j})\] is a decomposition of $T$ and in particular \[\sr T \le r_1^0 + \dots + r_d^0.\] \end{theorem}

We stress that in the assumptions of Theorem \ref{At most d decompositions} the decompositions are not required to have minimal length. The inequality $h > d$ is optimal, as can be seen from either Construction \ref{Construction 2} or Construction \ref{Construction 3}.

We note that there is a large class of examples where the subspaces $A_1^\t, \dots, A_d^\t$ are respectively the same for every minimal-length decomposition of $T$. Informally, this is the class of order-$d$ tensors which admit a minimal-length slice rank decomposition where the $(d-1)$-variable functions of each kind are sufficiently separated with respect to the slice rank of order-$(d-1)$ tensors.

\begin{proposition}\label{Example with same subspaces} Let $d \ge 3$, $k \ge 1$ be integers, and let $T$ be an order-$d$ tensor with slice rank $k$. Assume that there exist nonnegative integers $r_1, \dots, r_d$ satisfying \[r_1+\dots+r_d = k,\] and a decomposition \[\sum_{j=1}^d \sum_{i=1}^{r_j} a_{j,i}(x_j) b_{j,i}(\overline{x_j})\] of $T$ such that for every $j \in [d]$ the order-$(d-1)$ tensors $b_{j,1}, \dots, b_{j,r_j}: \prod_{j’ \neq j} [n_j] \to \F$ satisfy \begin{equation} \sr \left( \sum_{i=1}^{r_j} \lambda_i b_{j,i} \right) \ge 2k \label{Lower bound on order-(d-1) slice rank} \end{equation} for every $(\lambda_1, \dots, \lambda_{r_j}) \in \F^{r_j} \setminus \{0\}$ and for every $j \in [d]$. Then the linear subspaces \[A_j = \langle a_{j,i}: i \in [r_j] \rangle\] with $j \in [d]$ are such that whenever \[\sum_{j=1}^d \sum_{i=1}^{r_j’} a_{j,i}’(x_j) b_{j,i}’(\overline{x_j})\] is a slice rank decomposition of $T$ satisfying $r_1’+\dots+r_d’ = k$, we have $r_j’=r_j$ and \[A_j = \langle a_{j,i}': i \in [r_j] \rangle \] for every $j \in [d]$. \end{proposition}

\begin{proof}

For every $j \in [d]$ we write $A_j$ and $A_j'$ for the linear subspaces \[\langle a_{j,1}, \dots, a_{j,r_j} \rangle \text{ and } \langle a_{j,1}, \dots, a_{j,r_j} \rangle\] respectively. If for some $j_0 \in [d]$ the subspace $A_{j_0}'$ does not contain $A_{j_0}$, then there exists a function $u:[n_{j_0}] \to \F$ such that $u.a = 0$ for every $a \in A_{j_0}'$ but such that $u.a \neq 0$ for some $a \in A_{j_0}$. Applying $u$ to both sides of the equality \[\sum_{j=1}^d \sum_{i=1}^{r_j} a_{j,i}(x_j) b_{j,i}(\overline{x_j}) = \sum_{j=1}^d \sum_{i=1}^{r_j'} a_{j,i}'(x_j) b_{j,i}'(\overline{x_j})\] we obtain \[\sum_{i=1}^{r_{j_0}} (u.a_{j_0,i}) b_{j_0,i}(\overline{x_{j_0}}) + \sum_{j \neq j_0} \sum_{i=1}^{r_j} a_{j,i}(x_j) (u.b_{j,i})(\overline{x_{j_0}, x_j}) = \sum_{j \neq j_0} \sum_{i=1}^{r_j'} a_{j,i}'(x_j) (u.b_{j,i}')(\overline{x_{j_0}, x_j}).\] Our assumption \eqref{Lower bound on order-(d-1) slice rank} shows that \[(u.a_{j_0,1}, \dots, u.a_{j_0, r_{j_0}}) \neq 0\] and hence that the first sum of the left-hand side has slice rank at least $2k$. All inner summands of all other sums each have slice rank at most $1$, and there are at most $2k-1$ such summands, which is a contradiction. Therefore, $A_{j}'$ contains $A_{j}$ for every $j \in [d]$. Since \[r_1 + \dots + r_d = k = r_1' + \dots + r_d',\] we have $A_j' = A_j$ for every $j \in [d]$. \end{proof}

From Proposition \ref{order-$d$ decompositions of zero} and Theorem \ref{At most d decompositions} we then deduce our theorem on the structure of minimal-length slice rank decompositions of tensors, which we now state.

\begin{theorem}\label{Main theorem} Let $d \ge 3$, $k \ge 1$ be integers, let $\F$ be a field, and let $T$ be an order-$d$ tensor over $\F$ and with slice rank equal to $k$. Then we have the following. \begin{enumerate}
\item If the field $\F$ is finite, then there exists a set $\ca$ of $d$-tuples $(W_1, \dots, W_d)$ of linear subspaces of $\F^{n_1}, \dots, \F^{n_d}$ respectively, with size $|\ca| \le d^k |\F|^{dk^2}$, such that if $r_1, \dots, r_d$ are nonnegative integers satisfying \[r_1+\dots+r_d=k\] and $a_{j,i}: [n_j] \to \F$ with $j \in [d]$ and $i \in [r_j]$ are functions satisfying \[T(x_1, \dots, x_d) = \sum_{j=1}^d \sum_{i=1}^{r_j} a_{j,i}(x_j) b_{j,i}(\overline{x_j})\] for some $(d-1)$-variable functions $b_{j,i}: \prod_{j' \neq j} [n_{j'}] \to \F$, then \[(\langle a_{1,1},\dots,a_{1,r_1} \rangle, \dots, \langle a_{d,1},\dots,a_{d,r_d} \rangle) \in \ca.\] In particular, there are at most $d^k |\F|^{2d k^2}$ possibilities in total for \[((r_1, \dots, r_d), (a_{1,1}, \dots, a_{1,r_1}), \dots, (a_{d,1}, \dots, a_{d,r_d})).\]
\item For any field $\F$, if two decompositions \begin{equation} T(x_1, \dots, x_d) = \sum_{j=1}^d \sum_{i=1}^{r_j} a_{j,i}(x_j) b_{j,i}^1(\overline{x_j}) = \sum_{j=1}^d \sum_{i=1}^{r_j} a_{j,i}(x_j) b_{j,i}^2(\overline{x_j}) \label{two decompositions} \end{equation} of $T$ have the same one-variable functions $a_{j,i}$, then the decomposition \begin{equation} \sum_{j=1}^d \sum_{i=1}^{r_j} a_{j,i}(x_j) (b_{j,i}^2 - b_{j,i}^1)(\overline{x_j}) \label{difference of decompositions} \end{equation} is in zero form.\end{enumerate} \end{theorem}

In the second item of Theorem \ref{Main theorem}, the decompositions need not be assumed to have minimal length, even if like the first item it was conceived primarily with this case in mind. The following example shows that the exponents in the two bounds from the first item of Theorem \ref{Main theorem} can only be away from the optimal bounds by factors which are linear in $d$, uniformly in $k$ and $\F$.

\begin{example}\label{matrix times slice rank 2r example} \emph{Let $d \ge 4$ be an integer, let $r$ be a positive integer, let $n_1,\dots,n_d$ be large integers compared to $r$, let $M: [n_1] \times [n_2] \to \F$ be a rank-$2r$ matrix with a decomposition \[M(x_1,x_2) = \sum_{i=1}^{2r} a_i^1(x) a_i^2(y), \] let $c:[n_3] \times \dots \times [n_d] \to \F$ be a function, and let $T$ be the order-$d$ tensor defined by \[T(x_1,\dots, x_d) = M(x_1,x_2) c(x_3, \dots, x_d).\] If the function $c$ has slice rank at least 2r as an order-(d-2) tensor then the tensor $T$ has slice rank $2r$ (if instead $d=3$, then $T$ always has slice rank at most $1$). Then there are at least $\omega |\F|^{r^2}$ possibilities for the linear subspace $\langle a_1^1, \dots, a_r^1 \rangle$.}\end{example}

\begin{proof}

The decompositions \[ \sum_{i=1}^r a_i^1(x_1) (a_i^2 c)(x_2,x_3,\dots,x_d) + \sum_{i=r+1}^{2r} a_i^2(x_2) (a_i^1 c)(x_1,x_3,\dots,x_d)\] are decompositions of $T$ with length $2r$, and the linear subspace $\langle a_1^1, \dots, a_r^1 \rangle$ can be any dimension-$r$ linear subspace of $\langle a_1^1, \dots, a_{2r}^1 \rangle$. In turn, there are at least at least \[\omega |\F|^{(2r)^2} / |\F|^{2r^2} |\F|^{r^2}\] possibilities for this linear subspace, which provides the desired bound. It hence suffices to check that these decompositions indeed have minimal length, in other words that T has slice rank no less than 2r. Assume for contradiction that \[\sum_{i=1}^{r_1} a_{1,i}(x_1) b_{1,i}(\overline{x_1}) + \sum_{i=1}^{r_2} a_{2,i}(x_2) b_{2,i}(\overline{x_2}) + \sum_{j=3}^d \sum_{i=1}^{r_j} a_{j,i}(x_j) b_{j,i}(\overline{x_j})\] is a slice rank decomposition of $T$ satisfying \[r_1 + r_2 + r_3 + \dots + r_d < 2r.\] The linear subspace \[U=\{a_1^* \in \F^{n_1}: a_1^*.a_{1,1} = 0, \dots, a_1^*.a_{1,r_1} = 0\}\] has dimension $n-r_1$, so since $M$ has rank $2r$, the linear subspace \[\{a_1^*.M: a_1^* \in U\}\] has dimension at least $2r-r_1$, so is not contained inside the linear subspace $\langle a_{2,1}, \dots, a_{2,r_2} \rangle$. Therefore, we can find $a_1^* \in U$ and an element $a_2^*$ of \[\{a_2^* \in \F^{n_2}: a_2^*.a_{2,1} = 0, \dots, a_2^*.a_{2,r_2} = 0\}\] which furthermore satisfies \[(a_1^* \otimes a_2^*).M = a_2^*.(a_1^*.M) \neq 0.\] Applying $a_1^* \otimes a_2^*$ to the original decomposition of $T$ then provides a slice rank decomposition \[\l c(x_3, \dots, x_d)= \sum_{j=3}^d \sum_{i=1}^{r_j} a_{j,i}(x_j) ((a_1^* \otimes a_2^*).b_{j,i})(\overline{x_1,x_2,x_j})\] for some $\l \neq 0$, which contradicts that $\sr c \ge 2r$.\qedhere \end{proof}

\section{Deduction of the structure theorem on slice rank decompositions}\label{Section: Deduction of the structure theorem on slice rank decompositions}

In this section we deduce Theorem \ref{Main theorem} from Proposition \ref{order-$d$ decompositions of zero} and Theorem \ref{At most d decompositions}.

\begin{proof}[Proof of Theorem \ref{Main theorem}]

Let us begin by proving the second item. Because the two decompositions \eqref{two decompositions} are decompositions of the same tensor, their difference \eqref{difference of decompositions} is equal to the zero tensor, so it follows from Proposition \ref{order-$d$ decompositions of zero} that the decomposition \eqref{difference of decompositions} is in zero form.

There remains to prove the first item. Let $k$ be the slice rank of $T$. If $k=0$ then we are done, so let us assume $k \ge 1$. We consider an arbitrary set of minimal-length decompositions \[\sum_{j=1}^d \sum_{i=1}^{r_j^\t} a_{j,i}^{\t}(x_j) b_{j,i}^{\t}(\overline{x_j}) \] of $T$ indexed by $\t \in [h]$ for some positive integer $h$. For every $\t \in [h]$ we write $A_1^\t, \dots, A_d^\t$ for the linear subspaces spanned by the functions $a_{1,i}^\t, \dots, a_{d,i}^\t$ respectively. We assume that whenever $\t,\t'$ are distinct elements of $[h]$, the equality $A_j^\t = A_j^{\t'}$ fails for at least one $j \in [d]$.

Let $M_1$ be a maximal subset of $[h]$ such that for each $j \in [d]$ the linear subspaces $A_j^\t$ with $\t \in M_1$ are all in direct sum. By Theorem \ref{At most d decompositions} the set $M_1$ must have size at most $d$, as otherwise we would have $T=0$. We can hence write

\[[h] = \bigcup_{1 \le j_1 \le d} \bigcup_{w_1} \Theta_{(j_1,w_1)}\] where the union over $w_1$ is taken over all non-zero vectors of $\oplus_{\t \in M_1} A_{j_1}^\t$, and where for every $(j_1,w_1)$ and every $\t \in M_1$ the line containing $w_1$ is contained in $A_{j_1}^{\t}$.

We now fix $(j_1,w_1)$, and let $M_2$ be a maximal subset of $\Theta_{(j_1,w_1)}$ such that the linear subspaces $A_{j_1}^{\t}$ with $\t \in M_2$ have pairwise intersection $\langle w_1 \rangle$, and such that for every $j \neq j_1$ the linear subspaces $A_{j}^{\t}$ with $\t \in M_2$ are in direct sum. Again, by applying Theorem \ref{At most d decompositions} the set $M_2$ has size at most $d$, as otherwise we would have $\sr T \le 1$, and we can further write \[\Theta_{(j_1,w_1)} = \bigcup_{1 \le j_2 \le d} \bigcup_{w_2} \Theta_{(j_1,w_1),(j_2,w_2)}\] where the union over $w_2$ is taken over all non-zero vectors of $\oplus_{\t \in M_1} A_{j_2}^\t$ if $j_2 \neq j_1$ (resp. taken over all vectors linearly independent from $w$ if $j_2=j_1$), and where for every $\t \in \Theta_{(j_2,w_2)}$ the line containing $w_2$ is contained in $A_{j_2}^\t$ (resp. the plane containing $w_1,w_2$ is contained in $A_{j_1}^\t$).

Iterating further we obtain a tree structure with depth $k$, and with root $[h]$. Let us describe the inductive step of the process. Once the tree is constructed up to some depth $\kappa \le k$, a (depth-$\kappa$) leaf of the tree constructed at this point is a subset \[\Theta_{(j_1,w_1),(j_2,w_2), \dots, (j_\kappa,w_\kappa)}\] of $[h]$ such that there exist linear subspaces $W_1 \subset \F^{n_1} \dots, W_d \subset \F^{n_d}$ with \[\dim W_1 + \dots + \dim W_d = \kappa\] and such that $W_j \subset A_j^{\t}$ for every $j \in [d]$ and every $\t \in \Theta_{(j_1,w_1),(j_2,w_2), \dots, (j_\kappa,w_\kappa)}$. If $\kappa<k$, then let $M_{\kappa+1}$ be a maximal subset of $\Theta_{(j_1,w_1),(j_2,w_2), \dots, (j_\kappa,w_\kappa)}$ such that the linear subspaces $A_j^{\t}$ with $\t \in M_{\kappa+1}$ have pairwise intersection $W_j$ for every $j \in [d]$. Theorem \ref{At most d decompositions} then shows that $M_{\kappa+1}$ has size at most $d$ (as otherwise we would have $\sr T \le \kappa)$, and we can hence write \begin{equation} \Theta_{(j_1,w_1),(j_2,w_2), \dots, (j_\kappa,w_\kappa)} = \bigcup_{1 \le j_{\kappa+1} \le d} \bigcup_{w_{\kappa+1}} \Theta_{(j_1,w_1),(j_2,w_2), \dots, (j_\kappa,w_\kappa), (j_{\kappa+1}, w_{\kappa+1})}, \label{Union at the general step} \end{equation} where the union over $w_{\kappa+1}$ is over all vectors of \[(+_{\t \in M_{\kappa+1}} A_{j_{\kappa+1}}^\t) \setminus W_{j_{\kappa+1}}.\] We then take the sets $\Theta_{(j_1,w_1),(j_2,w_2), \dots, (j_\kappa,w_\kappa), (j_{\kappa+1}, w_{\kappa+1})}$ to be the immediate descendants of the set $\Theta_{(j_1,w_1),(j_2,w_2), \dots, (j_\kappa,w_\kappa)}$ in the tree.

If instead $\kappa=k$ then $\dim W_1 + \dots + \dim W_d$ and $\dim A_1^\t + \dots + \dim A_d^\t$ are both equal to $k$, so are equal to one another. Hence, once we reach depth $k$ the linear subspaces $A_1^\t, \dots, A_d^\t$ are all completely determined by $W_1, \dots, W_d$. By our initial assumption that for any distinct $\t, \t' \in [h]$ we have $A_j^{\t} \neq A_j^{\t'}$ for some $j \in [d]$, the sets $\Theta_{(j_1,w_1),(j_2,w_2), \dots, (j_\kappa,w_\kappa)}$ each have size at most $1$.

We have obtained a tree with depth $k$, and each node of the tree has at most $d |\F|^{dk}$ immediate descendants, so the number of leaves, and hence of $d$-tuples of linear subspaces $(A_1^\t, \dots, A_d^\t)$ is at most $(d |\F|^{dk})^k = d^k |\F|^{dk^2}$. Furthermore, for each choice of a $d$-tuple of subspaces $(A_1^{\t}, \dots, A_d^{\t})$ with respective dimensions $r_1, \dots, r_d$, by the number \eqref{number of decompositions of matrices} of decompositions of matrices there are at most \[|\F|^{d r_1^2} \dots |\F|^{d r_d^2} \le |\F|^{d k^2}\] choices for \[((a_{1,1}, \dots, a_{1,r_1}), \dots, (a_{d,1}, \dots, a_{d,r_d})).\] This finishes the proof of the first item of Theorem \ref{Main theorem}. \qedhere \end{proof}

\section{Simpler analogue for tensor rank decompositions}\label{Section: the tensor rank case}

In this section we prove a simpler analogue of Theorem \ref{Main theorem} in the case of the tensor rank, which is much more similar to the statement for matrices that we discussed in the introduction.

\begin{proposition}\label{tensor rank case}

Let $d \ge 2$ be an integer, let $k$ be a nonnegative integer, and let $\F$ be a field. If $T: [n_1] \times \dots \times [n_d] \to \F$ is an order-$d$ tensor with tensor rank $k$, then there exist linear subspaces $A_1 \subset \F^{n_1}, \dots, A_d \subset \F^{n_d}$ such that if \begin{equation}T(x_1, \dots,x_d) = \sum_{i=1}^k a_{1,i}(x_1) \dots a_{d,i}(x_d) \label{tensor rank decomposition at the start of the tensor rank section} \end{equation} is a tensor rank decomposition of $T$ with length $k$, then we have \[A_1 = \langle a_{1,1}, \dots, a_{1,k} \rangle, \dots, A_d = \langle a_{d,1}, \dots, a_{d,k} \rangle.\] In particular, if the field $\F$ is finite, then the number of possible tensor rank decompositions of $T$ with length $k$ is at most $|\F|^{(d-1)k^2}$.

\end{proposition}

\begin{proof}

Assume that \[T(x_1,\dots ,x_d) = \sum_{i=1}^k a_{1,i}(x_1)\dots a_{d,i}(x_d) = \sum_{i=1}^k a_{1,i}'(x_1)\dots a_{d,i}'(x_d)\] are two decompositions of $T$ with length $k$. To prove the first conclusion, it suffices to show that for each $j \in [d]$ the linear subspaces $\langle a_{j,1},\dots ,a_{j,k} \rangle$ and $\langle a_{j,1}',\dots ,a_{j,k}' \rangle$ are the same. Assume for contradiction that they are not. Then without loss of generality the linear subspace $\langle a_{1,1},\dots ,a_{1,k} \rangle$ is not contained in the linear subspace $\langle a_{1,1}',\dots ,a_{1,k}' \rangle$, so there exists $u:[n_1] \to \F$ such that $(u.a_{1,1}', \dots , u.a_{1,k}') = 0$ but $(u.a_{1,1}, \dots , u.a_{1,k}) \neq 0$. Applying $u$ to both decompositions of $T$ we obtain

\[\sum_{i=1}^k (u.a_{1,i}) a_{2,i}(x_2)\dots a_{d,i}(x_d) = 0,\] so the products $a_{2,i} \otimes \dots \otimes a_{d,i}$ with $i \in [k]$ are linearly dependent. Assuming without loss of generality that the product $a_{2,k} \otimes \dots \otimes a_{d,k}$ is a linear combination of the products $a_{2,i} \otimes \dots \otimes a_{d,i}$ with $i \in [k-1]$, we can hence write

\[T(x_1,\dots ,x_d) = \sum_{i=1}^{k-1} \a_{1,i}(x_1) a_{2,i}(x_2) \dots a_{d,i}(x_d)\] for some new functions $\a_{1,1},\dots ,\a_{1,k-1}: [n_1] \to \F$. This is a tensor rank decomposition of $T$ with length $k-1$, which contradicts that $T$ has tensor rank $k$.

If the field $\F$ is finite, then for each of the $k(d-1)$ functions $a_{j,i}$ with $j \in \{2,\dots,d\}$ and $i \in [k]$ involved in a given decomposition of $T$, there are at most $|A_j| \le |\F|^k$ possibilities. Once these functions are fixed, the linear independence of the products \[a_{2,i} \otimes \dots \otimes a_{d,i}\] with $i \in [k]$, shows that the functions $a_{1,1}, \dots, a_{1,k}$ are determined. The bound follows. \end{proof}

Although the upper bound that we obtain specialises in the case $d=2$ to the upper bound obtained at the start of the introduction, the number of tensor rank decompositions can be much less than that for general $d$. Indeed if an order-$d$ tensor $T$ with tensor rank $k$ is identifiable in the sense that we mentioned in the introduction, then it can be written in at most $k!$ ways as a sum of $k$ tensors with tensor rank $1$ (those being the same, up to permutation); as in turn there are $(|\F|-1)^{d-1}$ ways of rewriting a tensor with tensor rank $1$ as a product of $d$ functions, such as tensor has at most $k! (|\F|-1)^{d-1}$ choices for the number of possible tensor rank decompositions in our sense. Nonetheless, the following example shows that the behaviour from the $d=2$ case for $k$ large can carry over to tensors of higher order, and that the exponent in the bound from Proposition \ref{tensor rank case} cannot be improved by a factor greater than $d$ in general. We recall the absolute constant $\omega = \prod_{i=1}^{\infty} (1-2^{-i})>0$ defined at the start of the introduction.

\begin{example} \emph{Let $d \ge 2$ be an integer, and let $k$ be a nonnegative integer. Let $T$ be the order-$d$ tensor defined by \begin{equation} T(x_1, \dots, x_d) = M(x_1,x_2) a_3(x_3) \dots a_d(x_d) \label{tensor rank decomposition of the example} \end{equation} for some matrix $M: [n_1] \times [n_2] \to \F$ with rank $k$ and some non-zero functions $a_3: [n_3] \to \F, \dots, a_d: [n_d] \to \F$. Then $T$ has tensor rank equal to $k$, and the number of tensor rank decompositions of length $k$ of $T$ is exactly \[(|\F|-1)^{(d-2)k} |\F|^{k^2} \prod_{i=1}^{k} (1-|\F|^{-i}),\] which is between $\omega |\F|^{k^2}$ and $|\F|^{k^2 + (d-2)k}$.} \end{example}

\begin{proof}

In one direction, writing a rank decomposition of the matrix $M$ with length $k$ and plugging it in \eqref{tensor rank decomposition of the example} shows $\tr T \le k$. In the other direction, if \begin{equation} T(x_1, \dots, x_d) = \sum_{i=1}^{k'} a_{1,i}(x_1) a_{2,i}(x_2) a_{3,i}(x_3) \dots a_{d,i}(x_d) \label{new tensor rank decomposition of the example} \end{equation} is a tensor rank decomposition of $T$ for some integer $k'$, then letting $a_3^*: [n_3] \to \F, \dots, a_d^*:[n_d] \to \F$ such that $a_j^*.a_j = 1$ for each $j \in \{3, \dots, d\}$ and applying $a_3^*, \dots, a_d^*$ to both decompositions \eqref{tensor rank decomposition of the example} and \eqref{new tensor rank decomposition of the example} leads to \[M(x_1,x_2) = \sum_{i=1}^{k'} C_i a_{1,i}(x_1) a_{2,i}(x_2)\] for some coefficients $C_i \in \F$, so $k' \ge \rk M$ and hence $\tr T \ge k$.

The decomposition \eqref{tensor rank decomposition of the example} shows that the linear subspaces $A_3, \dots, A_d$ of Proposition \ref{tensor rank case} each have dimension $1$, and there are hence at most $(|\F|-1)^{(d-2)k}$ choices for \begin{equation} ((a_{3,1}, \dots, a_{3,k}), \dots, (a_{d,1}, \dots, a_{d,k})) \label{one-variable functions between 3 and d} \end{equation} since if one of these functions were zero, then this would contradict $\tr T = k$. Once this choice is made, there are \[|\F|^{k^2} \prod_{i=1}^{k} (1-|\F|^{-i})\] choices for $(a_{2,1}, \dots, a_{2,k})$, and then $(a_{1,1}, \dots, a_{1,k})$ is completely determined. Conversely, for any given choice of \eqref{one-variable functions between 3 and d} such that $a_{j,i} \in A_j \setminus \{0\}$ for every $j \in \{3, \dots, d\}$ and $i \in [k]$, the decomposition \eqref{tensor rank decomposition at the start of the tensor rank section} becomes \[\sum_{i=1}^k D_i a_{1,i}(x_1) a_{2,i}(x_2) a_3(x_3) \dots a_d(x_d)\] for some $D_1, \dots, D_k \in \F \setminus \{0\}$, so whenever $f_1, \dots, f_k: [n_1] \to \F$, $g_1, \dots, g_k: [n_2] \to \F$ provide a decomposition \[M(x_1,x_2) = \sum_{i=1}^k f_i(x_1) g_i(x_2)\] of $M$ with length $k$, taking $a_{1,i} = f_i$ and $a_{2,i} = g_i/D_i$ for every $i \in [k]$ provides a decomposition of $T$ with length $k$. \end{proof}

\section{Proof of the theorem on sunflowers for order-$3$ tensors}\label{Section: The proof for order-$3$ tensors}

The main goal of this section is to prove Proposition \ref{order-$d$ decompositions of zero} and Theorem \ref{At most d decompositions} in the case of order-$3$ tensors, and discuss the main ideas behind the proofs as well as how they were discovered.

The statement of the next lemma is similar to Lemma 9 from \cite{Gowers}. The proof that we give is slightly different from the proof in that paper, but they are based on the same idea.

\begin{lemma} \label{lemma on every element spanned}

Let $F_A \subset \F^{n_1}, F_C \subset \F^{n_2}$ both be families of linearly independent vectors, and let $E$ be a linear subspace of $\F^{n_3}$. For each $(a,c) \in F_A \times F_C$, let $p_{a,c}$ be an element of $\F^{n_3}$. Assume that the linear combination \[\sum_{a \in F_A, c \in F_C} a \otimes c \otimes p_{a,c}\] belongs to $\F^{n_1} \otimes \F^{n_2} \otimes E$. Then $p_{a,c} \in E$ for every $(a,c) \in F_A \times F_C$. \end{lemma}

\begin{proof}

We define dual families $F_A^*$, $F_C^*$ to the families $F_A$, $F_C$ respectively. The assumption states that we can write \[\sum_{a \in A, c \in C} a \otimes c \otimes p_{a,c} = \sum_{i \in [n_1]} \sum_{j \in [n_2]} \tau_{i,1} \otimes \tau_{j,2} \otimes e_{i,j}\] where the vectors $\tau_{i,1}, \tau_{j,2}$ are respectively the canonical basis vectors of $\F^{n_1}$ and $\F^{n_2}$ respectively, and $e_{i,j} \in E$ for every $(i,j) \in [n_1] \times [n_2]$. For any given pair $(a,c) \in F_A \times F_C$, applying $a^* \otimes c^*$ provides \[p_{a,c} = \sum_{i=1}^{n_1} \sum_{j=1}^{n_2} a^*(i) c^*(j) e_{i,j}.\] This establishes the lemma. \end{proof}

We first prove Proposition \ref{order-$d$ decompositions of zero} in the $d=3$ case. It can be restated as follows, in a way that highlights the successive functions that will appear in the proof. As much of the proof is similar to that of Lemma 10 from \cite{Gowers} and to the discussion between that lemma and Corollary 11 there, we have decided to re-use the same notations as those of that paper.

\begin{proposition} \label{order-$3$ decompositions of zero} Let $r,s,t$ be nonnegative integers. Assume that we have the decomposition \begin{equation} \sum_{i=1}^r a_i(x) b_i(y,z) + \sum_{j=1}^s c_j(y) d_j(x,z) + \sum_{k=1}^t e_k(z) f_k(x,y) \label{decomposition equal to 0 for order-3 tensors in proof of structure of the order-2 functions} \end{equation} of the order-$3$ zero tensor $[n_1] \times [n_2] \times [n_3] \to \F$ and that the families \[(a_i: i \in [r]), (c_j: j \in [s]), (e_k: k \in [t])\] are each linearly independent. Then we may write \begin{align} b_i & = \sum_{j=1}^s c_j \otimes p_{i,j} + \sum_{k=1}^t q_{i,k} \otimes e_k \label{first equation for order-2 functions of order-3 tensors}\\ 
d_j & = \sum_{i=1}^r a_i \otimes g_{i,j} + \sum_{k=1}^t h_{j,k} \otimes e_k \label{second equation for order-2 functions of order-3 tensors}\\
f_k & = \sum_{i=1}^r a_i \otimes u_{i,k} + \sum_{j=1}^s v_{j,k} \otimes c_j  \label{third equation for order-2 functions of order-3 tensors}\end{align}
for each $i \in [r], j \in [s], k \in [t]$ respectively, where $h_{j,k}, v_{j,k}: [n_1] \to \F$, $q_{i,k},u_{i,k}: [n_2] \to \F$, $p_{i,j}, g_{i,j}: [n_3] \to \F$ are some functions satisfying \begin{align} p_{i,j} + g_{i,j} & = \sum_{k=1}^t \lambda_{i,j,k} e_k \label{linear combination of the e_k} \\
q_{i,k} + u_{i,k} & = \sum_{j=1}^s \mu_{i,j,k} \label{linear combination of the c_j} c_j\\
h_{j,k} + v_{j,k} & = \sum_{i=1}^r \nu_{i,j,k} a_i \label{linear combination of the a_i} \end{align} for each $ (i,j) \in [r] \times [s]$, for each $ (i,k) \in [r] \times [t]$, and for each $ (j,k) \in [s] \times [t]$ respectively, with $\lambda_{i,j,k}, \mu_{i,j,k}, \nu_{i,j,k} \in \F$ such that \begin{equation} \lambda_{i,j,k} + \mu_{i,j,k} + \nu_{i,j,k} = 0 \label{sum of three coefficients equal to 0} \end{equation} for each $(i,j,k) \in [r] \times [s] \times [t]$. \end{proposition}

\begin{proof}

Applying dual functions $a_i^*, c_j^*, e_k^*$ we get the expressions \eqref{first equation for order-2 functions of order-3 tensors}, \eqref{second equation for order-2 functions of order-3 tensors}, \eqref{third equation for order-2 functions of order-3 tensors}. Substituting in \eqref{decomposition equal to 0 for order-3 tensors in proof of structure of the order-2 functions} and applying Lemma \ref{lemma on every element spanned} shows that \begin{align*} p_{i,j} + g_{i,j} &\in \langle e_k: k \in [t] \rangle \\q_{i,k} + u_{i,k} &\in \langle c_j: j \in [s] \rangle\\h_{j,k} + v_{j,k} &\in \langle a_i: i \in [r] \rangle
\end{align*} respectively for each $(i,j) \in [r] \times [s]$, for each $(i,k) \in [r] \times [t]$, and for each $(j,k) \in [s] \times [t]$, so can write these functions as in \eqref{linear combination of the e_k}, \eqref{linear combination of the c_j}, \eqref{linear combination of the a_i}. Substituting in \eqref{decomposition equal to 0 for order-3 tensors in proof of structure of the order-2 functions} we obtain \[\sum_{i=1}^r \sum_{j=1}^s \sum_{k=1}^t (\lambda_{i,j,k}+\mu_{i,j,k}+\nu_{i,j,k}) a_i \otimes c_j \otimes e_k = 0.\] Because the products $a_i \otimes c_j \otimes e_k$ with $ (i,j,k) \in [r] \times [s] \times [t]$ are linearly independent, we conclude \eqref{sum of three coefficients equal to 0}. \end{proof}

Now turning to the one-variable functions, we can begin by asking for the weakest possible property asserting some lack of independence between the one-variable functions involved in several decompositions. As the strongest form of independence is to have, for every fixed $j \in [d]$, the linear subspaces $A_j^{\t}$ spanned by the functions $a_{j,i}^{\t}$ of the decompositions (indexed by $\t$) all be in direct sum, we aim to show that this can only occur for a bounded number of decompositions.

As formulated this last goal is not yet completely reasonable, as shown by the fact that the zero tensor can be decomposed as \[a(x) (bc)(y,z) - b(y) (ac)(x,z)\] for all $a: [n_1] \to \F$, $b: [n_2] \to \F$, $c: [n_3] \to \F$ (or even more simply, by repeating the length-$0$ decomposition of the zero tensor arbitrarily many times). What we will rather show is that the zero tensor is the only order-$3$ tensor for which there exist four decompositions where for each of the three choices of special coordinate the one-variable functions may all be in direct sum.

\begin{proposition}\label{A tensor with four independent decompositions is zero}

Let $T: [n_1] \times [n_2] \times [n_3] \to \F$ be an order-3 tensor. Assume that for every $\t \in [4]$ we have a decomposition \begin{equation} T(x,y,z) = \sum_{i=1}^{r_{\t}} a_i^\t(x) b_i^\t(y,z) + \sum_{j=1}^{s_\t} c_j^\t(y) d_j^\t(x,z) + \sum_{k=1}^{t_\t} e_k^\t(z) f_k^\t(x,y) 
\label{decomposition of T for each u} \end{equation} for some nonnegative integers $r_\t, s_\t, t_\t$. Assume that the families \begin{equation} (a_i^\t: \t \in [4], i \in [r_\t]) \text{, } (c_j^\t: \t \in [4], j \in [s_\t]) \text{, } (e_k^\t: \t \in [4], k \in [t_\t]) \label{linearly independent families for four decompositions} \end{equation} are each linearly independent. Then $T=0$.

\end{proposition}

Let us briefly outline the proof before writing it formally. Writing an equality between two decompositions of $T$ (indexed by, say, $1$ and $2$) and applying the dual functions then allows us to write any given two-variable function of the first decomposition (such as, say, $b_1^{1}$) as a sum of products of two one-variable functions, one of which, of one of the types $c_j^1$, $c_j^2$, $e_k^1$, $e_k^2$, is specified and may come from either of the two decompositions, and the other of which is unspecified. We can write similar expressions of $b_1^{1}$ where the third and fourth decompositions supply the one-variable functions that the second decomposition previously supplied, which provides two other decompositions of $b_1^{1}$, after which it is only moderately difficult to show that $b_1^{1}$ has a decomposition of the kind that we have just described, where all specified one-variable functions are of the types $c_j^1$ and $e_k^1$. In other words we show that for every $\t \in [4]$ the two-variable functions $b_i^\t, d_j^\t, f_k^\t$ have decompositions as sums of products of two functions where at least one function of each product is of the type $a_i^\t, c_j^\t, e_k^\t$ for that same value of $\t$. This is the statement of Lemma \ref{Intermediate result on decompositions}, and is a property which is strong enough to allow us to conclude that $T=0$ once we apply it to two different values of $\t$.

For every decomposition $\t$ we respectively write $A^{\t}, C^{\t}, E^{\t}$ for the linear subspaces \[\langle a_i^\t: i \in [r_\t] \rangle \text{, } \langle c_j^\t: j \in [s_\t] \rangle \text{, } \langle e_k^\t: k \in [t_\t] \rangle.\] We begin with a lemma which informally says that under the assumptions of Proposition \ref{A tensor with four independent decompositions is zero}, the two-variable functions of every decomposition $\t$ can themselves be further decomposed as sums of products of two one-variable functions, where in every product one of the functions is a one-variable function from that same decomposition $\t$ of the original tensor. So far for we have used notations such as $a_i^\t$, $b_i^\t$ for functions arising from a decomposition $\t$. As will now write equalities between decompositions, and use them to obtain expressions for the functions $b_i^\t$, our calculations will involve terms that may depend either only on one decomposition or simultaneously on both of the decompositions between which we wrote the original equality. In the first case these terms will be denoted for instance by $p_{i,j}^{\t}$ to indicate that they only depend on the decomposition $\t$, and the second case they will instead be written for instance as $p_{i,j}^{\t \t'}$ to indicate that they depend on both decompositions $\t$,$\t'$.

\begin{lemma} \label{Intermediate result on decompositions} Let $T: [n_1] \times [n_2] \times [n_3] \to \F$ be an order-$3$ tensor. Under the assumptions of Proposition \ref{A tensor with four independent decompositions is zero}, for each $\t \in [4]$ we can write decompositions of the type \begin{align} b_i^\t(y,z) & = \sum_{j=1}^{s_\t} c_j^\t(y) p_{i,j}^\t(z) + \sum_{k=1}^{t_\t} q_{i,k}^\t(y) e_k^\t(z) \label{decomposition of b_i^u} \\
d_j^\t(x,z) & = \sum_{i=1}^{r_\t} a_i^\t(x) g_{i,j}^\t(z) + \sum_{k=1}^{t_\t} h_{j,k}^\t(x) e_k^\t(z) \label{decomposition of d_j^u}\\
f_k^\t(x,y) & = \sum_{i=1}^{r_\t} a_i^\t(x) u_{i,k}^\t(y) + \sum_{j=1}^{s_\t} v_{j,k}^\t(x) c_j^\t(y). \label{decomposition of f_k^u} \end{align}

\end{lemma}

\begin{proof}
We choose dual families $ (a_i^{\t*})$, $ (c_j^{\t*})$, $ (e_k^{\t*})$ to the respective (linearly independent) families \eqref{linearly independent families for four decompositions}. We begin with the equality \begin{align*} & \sum_{i=1}^{r_1} a_i^1(x) b_i^1(y,z) + \sum_{j=1}^{s_1} c_j^1(y) d_j^1(x,z) + \sum_{k=1}^{t_1} e_k^1(z) f_j^1(x,y) \\
= & \sum_{i=1}^{r_2} a_i^2(x) b_i^2(y,z) + \sum_{j=1}^{s_2} c_j^2(y) d_j^2(x,z) + \sum_{k=1}^{t_2} e_k^2(z) f_j^2(x,y). \end{align*} Applying the function $a_i^{1*}$ we get a decomposition of the type
\begin{equation} b_i^1(y,z) = \sum_{j=1}^{s_1} c_j^1(y) p_{i,j}^{1}(z) - \sum_{j=1}^{s_2} c_j^2(y) p_{i,j}^{12}(z) + \sum_{k=1}^{t_1} q_{i,k}^{1}(y) e_k^1(z) - \sum_{k=1}^{t_2} q_{i,k}^{12}(y) e_k^2(z). \label{previously equation 1}\end{equation} Here $2$ could instead have been taken to be an arbitrary element of $\{2,3,4\}$, so if we write the previous identity for the pair $(1,3)$ rather than for the pair $(1,2)$, and then subtract it from the identity for the pair $(1,2)$ then we get \[ \sum_{j=1}^{s_3} c_j^3(y) p_{i,j}^{13}(z) - \sum_{j=1}^{s_2} c_j^2(y) p_{i,j}^{12}(z) + \sum_{k=1}^{t_3} q_{i,k}^{13}(y) e_k^3(z) - \sum_{k=1}^{t_2} q_{i,k}^{12}(y) e_k^2(z) = 0.\] Applying the dual functions $c_j^{2*}$ then shows that \[p_{i,j}^{12} \in E^2 \oplus E^3\] for each $j \in [s_2]$. This argument also shows (by replacing $3$ by $4$) that \[p_{i,j}^{12} \in E^2 \oplus E^4\] for each $j \in [s_2]$. Because the linear subspaces $E^2, E^3, E^4$ are all in direct sum, we obtain \[p_{i,j}^{12} \in E^2\] for each $j \in [s_2]$. A similar argument shows that \[q_{i,k}^{12} \in C^2\] for each $k \in [t_2]$. Substituting in \eqref{previously equation 1} we obtain \[ b_i^1(y,z) = \sum_{j=1}^{s_1} c_j^1(y) p_{i,j}^{1}(z) + \sum_{k=1}^{t_1} q_{i,k}^{1}(y) e_k^1(z) + \sum_{j=1}^{s_2} \sum_{k=1}^{t_2} \lambda_{i,j,k}^{12} c_j^2(y) e_k^2(z) \] for some coefficients $\lambda_{i,j,k}^{12} \in \F$. Writing a similar equality with $3$ instead of $2$, subtracting it from the equality above, and using that the products $c_j^2 \otimes e_k^2$ and $c_j^3 \otimes e_k^3$ are all jointly independent we obtain that all coefficients $\lambda_{i,j,k}^{12}$ are equal to $0$, and hence the desired expression \eqref{decomposition of b_i^u}. We conclude \eqref{decomposition of d_j^u} and \eqref{decomposition of f_k^u} similarly. \end{proof}

Given Lemma \ref{Intermediate result on decompositions}, Proposition \ref{A tensor with four independent decompositions is zero} is now within reach.

\begin{proof}[Proof of Proposition \ref{A tensor with four independent decompositions is zero}]

Plugging in \eqref{decomposition of b_i^u}, \eqref{decomposition of d_j^u}, \eqref{decomposition of f_k^u} into \eqref{decomposition of T for each u} and gathering terms, we obtain \begin{align} T & = \sum_{i=1}^{r_1} \sum_{j=1}^{s_1} a_i^1 \otimes c_j^1 \otimes (p_{i,j}^1 + g_{i,j}^1) \nonumber \\
& + \sum_{i=1}^{r_1} \sum_{k=1}^{t_1} a_i^1 \otimes (q_{i,k}^1 + u_{i,k}^1) \otimes e_k^1 \nonumber \\
& + \sum_{j=1}^{s_1} \sum_{k=1}^{t_1} (h_{j,k}^1 + v_{j,k}^1) \otimes c_j^1 \otimes e_k^1 \label{Decomposition of T with factored sums} \end{align} and a similar decomposition with $2$ instead of $1$. Writing the equality between the two decompositions shows in particular that  \[\sum_{i=1}^{r_1} \sum_{j=1}^{s_1} a_i^1 \otimes c_j^1 \otimes (p_{i,j}^1 + g_{i,j}^1) - \sum_{i=1}^{r_2} \sum_{j=1}^{s_2} a_i^2 \otimes c_j^2 \otimes (p_{i,j}^2 + g_{i,j}^2)\] belongs to the linear subspace \[\F^{n_1} \otimes \F^{n_2}  \otimes (E^1 \oplus E^2).\] Because the products $a_i^1 \otimes c_j^1$ and $a_i^2 \otimes c_j^2$ are all linearly independent, Lemma \ref{lemma on every element spanned} shows that \[p_{i,j}^1 + g_{i,j}^1 \in E^1 \oplus E^2\] for each $(i,j) \in [r_1] \times [s_1]$. We obtain analogous statements for the sums \[q_{i,k}^1 + u_{i,k}^1 \text{, } h_{j,k}^1 + v_{j,k}^1 \text{, } p_{i,j}^2 + g_{i,j}^2 \text{, } q_{i,k}^2 + u_{i,k}^2 \text{, } h_{j,k}^2 + v_{j,k}^2.\] Substituting in the decomposition \eqref{Decomposition of T with factored sums} we obtain that $T$ belongs to the linear subspace \[ (A^1 \otimes C^1 \otimes E^1) \oplus (A^2 \otimes C^1 \otimes E^1) \oplus (A^1 \otimes C^2 \otimes E^1) \oplus (A^1 \otimes C^1 \otimes E^2),\] and similarly with the roles of $1$ and $2$ exchanged. The intersection of these two linear subspaces is $\{0\}$, so $T=0$. \end{proof}

The statement of Proposition \ref{A tensor with four independent decompositions is zero} does not appear to be strong enough to deduce Theorem \ref{Main theorem} for order-$3$ tensors. In particular, an unsuccessful attempt at a deduction runs as follows. Given any set of minimal-length decompositions of a tensor with slice rank $k$ (indexed by $\t \in [h]$) with size $h \ge 4$, Proposition \ref{A tensor with four independent decompositions is zero} shows that without loss of generality, every $\t \in [h]$ is such that at least one of the intersections \begin{equation} A^\t \cap (A^1 + A^2 + A^3) \text{, } C^\t \cap (C^1 + C^2 + C^3) \text{, } E^\t \cap (E^1 + E^2 + E^3) \label{non-empty intersection}\end{equation} is non-empty. Because the field $|\F|$ is assumed to be finite, we can write $[h]$ as a union of at most $3|\F|^{3k}$ sets, indexed by an element of [3] indicating one of the intersections \eqref{non-empty intersection} which is non-empty (chosen arbitrarily among the possibilities if more than one intersection is non-empty), and by a non-zero vector in this intersection. Restricting our attention to one of these sets, without loss of generality there exists $a: [n_1] \to \F$ such that all decompositions of $T$ are, after a change of basis in the first of the three parts of the decomposition, of the type \[a(x) b^\t(y,z) + \sum_{i=1}^{r_\t-1} a_i^\t(x) b_i^\t(y,z) + \sum_{j=1}^{s_\t} c_j^\t(y) d_j^\t(x,z) + \sum_{k=1}^{t_\t} e_k^\t(z) f_k^\t(x,y)\] with $r_\t+s_\t+t_\t = k$. The tensor $T^\t$ defined by \[T^\t(x,y,z) = T(x,y,z) - a(x) b^\t(y,z)\] must then have slice rank $k-1$ (if it were strictly less than that, then the subadditivity of the slice rank would show that $\sr T < k$). We could then aim to induct on the slice rank of $k$ by applying on each $T^\t$ the inductive hypothesis stating that the number of possibilities for the one-variable functions in minimal-length decompositions is bounded by some function of $d,k,\F$ only. The inductive step however does not work, because although the one-variable function $a$ is the same for every $\t$, the number of possibilities for the two-variable function $b^\t$ (and hence for the tensor $T^\t$) could potentially be unbounded.

To transform this attempt into the proof from Section \ref{Section: Deduction of the structure theorem on slice rank decompositions}, we want to prove a strengthening of Proposition \ref{A tensor with four independent decompositions is zero} which does not require us to use the auxiliary tensors $T^\t$ and instead allows us to work with decompositions of $T$ directly. A suitable formulation is given by the following statement, which is the special case $d=3$ of Theorem \ref{Main theorem}.

\begin{proposition}\label{A tensor with four common-independent decompositions has slice rank at most the common part}

Let $T: [n_1] \times [n_2] \times [n_3] \to \F$ be an order-3 tensor. Assume that for each $\t \in [4]$ we have a decomposition \begin{align} T(x,y,z) & = \sum_{i=1}^{r_0} a_i^0(x) b_i^{0,\t}(y,z) + \sum_{j=1}^{s_0} c_j^0(y) d_j^{0,\t}(x,z) + \sum_{k=1}^{t_0} e_k^0(z) f_k^{0,\t}(x,y) \nonumber \\ & + \sum_{i=1}^{r_\t} a_i^\t(x) b_i^\t(y,z) + \sum_{j=1}^{s_\t} c_j^\t(y) d_j^\t(x,z) + \sum_{k=1}^{t_\t} e_k^\t(z) f_k^\t(x,y) \label{common-independent decomposition of T for each u} \end{align} for some nonnegative integers $r_\t, s_\t, t_\t$. Assume that the families \begin{equation} (a_i^\b: \b \in \{0\} \cup [4], i \in [r_\b]) \text{, } (c_j^\b: \b \in \{0\} \cup [4], j \in [s_\b]) \text{, } (e_k^\b: \b \in \{0\} \cup [4], k \in [t_\b]) \label{linearly independent families for four sets with a center} \end{equation} are each linearly independent. Then $T$ has a decomposition of the type \[ \sum_{i=1}^{r_0} a_i^0(x) b_i(y,z) + \sum_{j=1}^{s_0} c_j^0(y) d_j(x,z) + \sum_{k=1}^{t_0} e_k^0(z) f_k(x,y)\] and in particular has slice rank at most $r_0+s_0+t_0$.

\end{proposition}

The bound of $r_0+s_0+t_0$ is sharp: if $T$ is a tensor with slice rank $k$, then it suffices to take the part with one-variable functions $a_i^0, c_j^0, e_k^0$ to be the same slice rank decomposition of $T$ with length $k$ for each $\t$, and to take $r_\t, s_\t, t_\t$ equal to $0$ for each $\t$.

Although the proof of Proposition \ref{A tensor with four common-independent decompositions has slice rank at most the common part} will in a way be slightly more complicated than the proof of Proposition \ref{A tensor with four independent decompositions is zero}, it will also in another way be more conceptually appealing: indeed, looking back at the argument about decompositions of two-variable functions which we had used in the proof of Proposition \ref{A tensor with four independent decompositions is zero}, we see that we had implicitly used and proved an analogue of Proposition \ref{A tensor with four common-independent decompositions has slice rank at most the common part} for $d=2$. This analogue, Proposition \ref{A matrix with three common-independent decompositions has slice rank at most the common part}, is the $d=2$ case of Theorem \ref{At most d decompositions} and furthermore lays the groundwork for the inductive argument on the order $d$ of the tensor that we will later use in Section \ref{Section: The proof in the general case} when dealing with the general case. Because the proof of Proposition \ref{A matrix with three common-independent decompositions has slice rank at most the common part} is contained in that of Proposition \ref{A tensor with four independent decompositions is zero}, now appears to be a good moment to pass to a slightly more abstract but more concise notation which we will use repeatedly in the remainder of the paper.

\begin{proposition} \label{A matrix with three common-independent decompositions has slice rank at most the common part}

Let $M: [n_1] \times [n_2] \to \F$ be a matrix. Assume that for each $\t \in [3]$ we have a decomposition \begin{align} M(x,y) & = \sum_{i=1}^{r_0} a_i^0(x) b_i^{0,\t}(y) + \sum_{j=1}^{s_0} d_j^{0,\t}(x) c_j^0(y) \nonumber \\ & + \sum_{i=1}^{r_\t} a_i^\t(x) b_i^\t(y) + \sum_{j=1}^{s_\t} d_j^\t(x) c_j^\t(y) \label{common-independent decomposition of A for each u} \end{align} for some nonnegative integers $r_\t, s_\t$. Assume that the families \begin{equation} (a_i^\b: \b \in \{0\} \cup [3], i \in [r_\b]) \text{, } (c_j^\b: \b \in \{0\} \cup [3], j \in [s_\b]) \label{linearly independent families for matrices} \end{equation} are both linearly independent. Then $M$ has a decomposition of the type \[ \sum_{i=1}^{r_0} a_i^0(x) b_i(y) + \sum_{j=1}^{s_0} d_j(x) c_j^0(y)\] and in particular has rank at most $r_0+s_0$.

\end{proposition}

\begin{proof}

For every $\b \in \{0\} \cup [3]$ we respectively write $A^{\b}$ and $C^{\b}$ for the linear subspaces $\langle a_i^\b: i \in [r_\b] \rangle$ and $\langle c_j^\b: j \in [s_\b] \rangle$. We define $(a_i^{\b}), (c_j^{\b})$ to be dual families to the families \eqref{linearly independent families for matrices}. Writing the equality between the decompositions $1$ and $2$ and applying the dual functions $a_i^{1*}$ we obtain $b_i^{1} \in C^0 \oplus C^1 \oplus C^2$. Similarly, $b_i^{1} \in C^0 \oplus C^1 \oplus C^3$. Because the subspaces $C^0 \oplus C^1$, $C^2$ and $C^3$ are in direct sum, we have $b_i^{1} \in C^0 \oplus C^1$. Similarly using the dual functions $c_j^{1*}$ we show $d_j^{1} \in A^0 \oplus A^1$. Using the decomposition $1$ of $M$ we obtain that \[ M \in A^0 \otimes \F^{n_2} + \F^{n_1} \otimes C^0 + A^1 \otimes (C^0 \oplus C^1) + (A^0 \oplus A^1) \otimes C^1,\] which simplifies to \[M \in (A^0 \otimes \F^{n_2} + \F^{n_1} \otimes C^0) \oplus A^1 \otimes C^1.\] Similarly the decomposition $2$ provides \[M \in (A^0 \otimes \F^{n_2} + \F^{n_1} \otimes C^0) \oplus A^2 \otimes C^2.\] The linear subspaces \[A^0 \otimes \F^{n_2} + \F^{n_1} \otimes C^0, A^1 \otimes C^1, A^2 \otimes C^2\] are all in direct sum, so \[M \in A^0 \otimes \F^{n_2} + \F^{n_1} \otimes C^0\] and the decomposition \eqref{common-independent decomposition of A for each u} follows. \end{proof}

Using Proposition \ref{A matrix with three common-independent decompositions has slice rank at most the common part} we are now able to obtain a generalisation of Lemma \ref{Intermediate result on decompositions} in rather short order. Throughout the remainder of this section, for every $\b \in \{0\} \cup [4]$ we respectively write $A^{\b}, C^{\b}, E^{\b}$ for the linear subspaces \[\langle a_i^\b: i \in [r_\b] \rangle \text{, } \langle c_j^\b: j \in [s_\b] \rangle \text{, } \langle e_k^\b: k \in [t_\b] \rangle,\] and define dual families $(a_i^{\b*})$,$(c_j^{\b*})$, $(e_k^{\b*})$ to the respective families \eqref{linearly independent families for four sets with a center}.

\begin{lemma} \label{Intermediate result on decompositions with common part} Let $T: [n_1] \times [n_2] \times [n_3] \to \F$ be an order-$3$ tensor. Under the assumptions of Proposition \ref{A tensor with four common-independent decompositions has slice rank at most the common part}, for each $\t \in [4]$ we can write decompositions of the type \begin{align} b_i^\t(y,z) & = \sum_{j=1}^{s_0} c_j^0(y) p_{i,j}^{0,\t}(z) + \sum_{k=1}^{t_0} q_{i,k}^{0,\t}(y) e_k^0(z) + \sum_{j=1}^{s_\t} c_j^\t(y) p_{i,j}^\t(z) + \sum_{k=1}^{t_\t} q_{i,k}^\t(y) e_k^\t(z) \label{decomposition of b_i^u in proposition with common part} \\
d_j^\t(x,z) & = \sum_{i=1}^{r_0} a_i^0(x) g_{i,j}^{0,\t}(z) + \sum_{k=1}^{t_0} h_{j,k}^{0,\t}(x) e_k^0(z) + \sum_{i=1}^{r_\t} a_i^\t(x) g_{i,j}^\t(z) + \sum_{k=1}^{t_\t} h_{j,k}^\t(x) e_k^\t(z)  \label{decomposition of d_j^u in proposition with common part}\\
f_k^\t(x,y) & = \sum_{i=1}^{r_0} a_i^0(x) u_{i,k}^{0,\t}(y) + \sum_{j=1}^{s_0} v_{j,k}^{0,\t}(x) c_j^0(y) + \sum_{i=1}^{r_\t} a_i^\t(x) u_{i,k}^\t(y) + \sum_{j=1}^{s_\t} v_{j,k}^\t(x) c_j^\t(y) \label{decomposition of f_k^u in proposition with common part}. \end{align}

\end{lemma}

\begin{proof}
Writing an equality between the decompositions $1,2$ of $T$ and applying the dual functions $a_i^{1*}$ we get a decomposition of the type \begin{align*} b_i^1(y,z) & = \sum_{j=1}^{s_0} c_j^0(y) (p_{i,j}^{0,1} - p_{i,j}^{0,12})(z) + \sum_{j=1}^{s_1} c_j^1(y) p_{i,j}^{1}(z) - \sum_{j=1}^{s_2} c_j^2(y) p_{i,j}^{12}(z)\\ & + \sum_{k=1}^{t_0} (q_{i,k}^{0,1} - q_{i,k}^{0,12})(y) e_k^0(z) + \sum_{k=1}^{t_1} q_{i,k}^{1}(y) e_k^1(z) - \sum_{k=1}^{t_2} q_{i,k}^{12}(y) e_k^2(z).\end{align*} We can write similar decompositions by replacing $2$ by $3$ or by $4$. We then apply Proposition \ref{A matrix with three common-independent decompositions has slice rank at most the common part}, with the pair of coordinates $(y,z)$ instead of $(x,y)$, with \[(c_j^\b: \b \in \{0,1\}, j \in [s_\b])\] forming a basis of center functions, \[(c_j^2: j \in [s_2]), (c_j^3: j \in [s_3]), (c_j^4: j \in [s_4])\] forming bases of petal functions, and similarly for the remaining coordinate.  The desired expressions \eqref{decomposition of b_i^u in proposition with common part}, \eqref{decomposition of d_j^u in proposition with common part}, \eqref{decomposition of f_k^u in proposition with common part} then follow. \end{proof}

There again, having established Lemma \ref{Intermediate result on decompositions with common part} we are ready to finish the proof of Proposition \ref{A tensor with four common-independent decompositions has slice rank at most the common part}.

\begin{proof}[Proof of Proposition \ref{A tensor with four common-independent decompositions has slice rank at most the common part}]

Writing the equality between the two decompositions \eqref{common-independent decomposition of T for each u} of $T$ for $\t=1$ and $\t=2$ and applying the dual functions $a_i^{0*}$ we obtain \begin{equation} b_i^{0,1} - b_i^{0,2} \in (C^0 \oplus C^1 \oplus C^2) \otimes \F^{n_3} + \F^{n_2} \otimes (E^0 \oplus E^1 \oplus E^2) \label{subspace for the difference b_i^{0,1} - b_i^{0,2}}. \end{equation} We obtain similar statements for the differences $d_j^{0,1} - d_j^{0,2}$ and $f_k^{0,1} - f_k^{0,2}$.

We define linear subspaces of $\F^{n_1} \otimes \F^{n_2} \otimes \F^{n_3}$, by \begin{align*} R_{\mathrm{target}} & = A^0 \otimes \F^{n_2} \otimes \F^{n_3} + \F^{n_1} \otimes C^0 \otimes \F^{n_3} + \F^{n_1} \otimes \F^{n_2} \otimes E^0\\
R_1 & = (A^1 \otimes C^1 \otimes \F^{n_3}) + (A^1 \otimes \F^{n_2} \otimes E^1) + (\F^{n_1} \otimes C^1 \otimes E^1)\\
R_2 & = (A^2 \otimes C^2 \otimes \F^{n_3}) + (A^2 \otimes \F^{n_2} \otimes E^2) + (\F^{n_1} \otimes C^2 \otimes E^2) \end{align*} and by letting $R$ be the sum of all linear subspaces of $\F^{n_1} \otimes \F^{n_2} \otimes \F^{n_3}$ of the type $Y_1 \otimes Y_2 \otimes Y_3$ satisfying at least one of the equalities \[Y_1 = A^0, Y_2 = C^0, Y_3 = C^0,\] and at least two of the properties \[Y_1 \in \{A^0, A^1, A^2 \} \text{, } Y_2 \in \{C^0, C^1, C^2 \} \text{, } Y_3 \in \{E^0, E^1, E^2 \}.\]

The tensor $S$ defined by \[S(x,y,z) = \sum_{i=1}^{r_0} a_i^0(x) b_i^{0,2}(y,z) + \sum_{j=1}^{s_0} c_j^0(y) d_j^{0,2}(x,z) + \sum_{k=1}^{t_0} e_k^0(z) f_k^{0,2}(x,y)\] belongs to $R_{\mathrm{target}}$. Subtracting this decomposition of $S$ from both sides of \eqref{common-independent decomposition of T for each u} we obtain an equality between two decompositions of $T-S$, which we can write as \begin{equation} T_1 + T_{\Delta} = T_2 \label{u,v and difference}, \end{equation} where \begin{align*} T_{\Delta} & = \sum_{i=1}^{r_0} a_i^0(x) (b_i^{0,1} - b_i^{0,2})(y,z) + \sum_{j=1}^{s_0} c_j^0(y) (d_j^{0,1} - d_j^{0,2})(x,z)+ \sum_{k=1}^{t_0} e_k^0(z) (f_k^{0,1} - f_k^{0,2})(x,y)\\
T_1 & = \sum_{i=1}^{r_1} a_i^1(x) b_i^1(y,z) + \sum_{j=1}^{s_1} c_j^1(y) d_j^1(x,z) + \sum_{k=1}^{t_1} e_k^1(z) f_k^1(x,y)\\
T_2 & = \sum_{i=1}^{r_2} a_i^2(x) b_i^2(y,z) + \sum_{j=1}^{s_2} c_j^2(y) d_j^2(x,z) + \sum_{k=1}^{t_2} e_k^2(z) f_k^2(x,y). \end{align*} From \eqref{subspace for the difference b_i^{0,1} - b_i^{0,2}} and the two other analogous identities for $d_j^{0,1} - d_j^{0,2}$ and $f_k^{0,1} - f_k^{0,2}$, the tensor $T_{\Delta}$ belongs to $R$. From Lemma \ref{Intermediate result on decompositions with common part} we know that \[T_1 \in R + R_1 \text{ and }T_2 \in R + R_2.\]

In the equality \eqref{u,v and difference} we gather terms of $T_{\Delta}, T_1, T_2$ which are of the same type, meaning that for instance if we have two terms \[a_i^1 \otimes c_j^0 \otimes p' \text{ and } a_i^1 \otimes c_j^0 \otimes p'',\] and $p'$,$p''$ are unspecified elements of $\F^{n_3}$, then we replace them by a single term \[a_i^1 \otimes c_j^0 \otimes (p'+p'') \text{ or } a_i^1 \otimes c_j^0 \otimes (p'-p'')\] as appropriate. We allow terms to change sides of the equation, and because we do so, the resulting equality is no longer between two decompositions of $T-S$, but between two decompositions of $T-S'$, for some new tensor $S'$. Any terms which appear on both sides of \eqref{u,v and difference} must necessarily have as a factor an element of one of the types $a_i^0$, $c_j^0$, $e_k^0$: indeed, if they do not, then their product of specified one-variable functions is necessarily of one of the types \[a_i^1 \otimes c_j^1, a_i^1 \otimes e_k^1, c_j^1 \otimes e_k^1, a_i^2 \otimes c_j^2, a_i^2 \otimes e_k^2, c_j^2 \otimes e_k^2,\] and in the case of the first three and last three the term can respectively only arise in the left-hand side and in the right-hand side. The difference $S'-S$ hence belongs to $R_{\mathrm{target}}$. Since $T-S$ also belongs to $R_{\mathrm{target}}$, we obtain that $T-S'$ belongs to $R_{\mathrm{target}}$.

Using Lemma \ref{lemma on every element spanned} and the linear independence of the families \eqref{linearly independent families for four sets with a center} we obtain that the terms involved in the resulting equality all belong to \[(A^0 \oplus A^1 \oplus A^2) \otimes (C^0 \oplus C^1 \oplus C^2) \otimes (E^0 \oplus E^1 \oplus E^2),\] and we can hence write both sides of the equality as linear combinations of terms of the type \begin{equation} a_i^{\b_1} \otimes c_j^{\b_2} \otimes e_k^{\b_3}. \label{product of three terms} \end{equation} The terms that have none of the indices $\b_1, \b_2, \b_3$ equal to $0$ can only come from a single side: the left-hand side if at least two indices $1$ are involved, the right-hand if at most one is. The linear independence of the families \eqref{linearly independent families for four sets with a center} (and hence of the products \eqref{product of three terms}) shows that the contribution of these terms is zero. We are left with an equality between terms which each have a factor of one of the types $a_i^0$, $c_j^0$, $e_k^0$, so either side belongs to $R_{\mathrm{target}}$. These sides are equal to $T-S'$, so $T-S'$ belongs to $R_{\mathrm{target}}$. Since $S'$ also does, we conclude that $T$ belongs to $R_{\mathrm{target}}$, and $T$ therefore has slice rank at most $r_0+s_0+t_0$, as desired. \end{proof}

\section{Proof of the theorem on sunflowers in the general case}\label{Section: The proof in the general case}

In this section we show Proposition \ref{order-$d$ decompositions of zero} and Theorem \ref{At most d decompositions} for general $d$, the proofs of which are based on and extend those of Proposition \ref{order-$3$ decompositions of zero} and Proposition \ref{A tensor with four common-independent decompositions has slice rank at most the common part} respectively. Whenever $X$ is a finite set and $m$ is a nonnegative integer, we write $\ct_m(X)$ for the set of $m$-tuples $(j_1, \dots, j_m)$ where $j_1, \dots, j_m$ are pairwise distinct elements of $X$.

In order to avoid excessively heavy notation we will for instance for a linear subspace $A_j$ of $\F^{n_j}$ write $A_j \otimes (\otimes_{j' \neq j} \F^{n_{j'}})$ for the tensor product that would more usually be written as \[(\otimes_{1 \le j' < j} \F^{n_{j'}}) \otimes A_j \otimes (\otimes_{j < j' \le d} \F^{n_{j'}}).\] However, this should not cause any confusion to the reader, because the dimension (in $[d]$) from which any term of any such product comes from will always be clearly identified.

We begin by proving Proposition \ref{order-$d$ decompositions of zero}. The key additional idea for this proof, compared to the $d=3$ case, is to gradually split terms of a slice rank decomposition, which originally are products of one specified one-variable function and one unspecified $(d-1)$-variable function, into products of $m$ specified one-variable functions and of one unspecified $(d-m)$-variable function, for $m=1,2,\dots,d-1$ before finally writing the decomposition as a product of $d$ specified one-variable functions. This process had been used in relation to the slice rank in the proof of Lemma 9 from the \emph{first} version of the manuscript \cite{Gowers}. Throughout the process, we collect terms together whenever they are of the same type, which may involve cancellations, but as these are all of the kind allowed in Definition \ref{zero form}, the definition of a tensor in zero form, we conclude that the original decomposition was in zero form.

\begin{proof}[Proof of Proposition \ref{order-$d$ decompositions of zero}] For every $j \in [d]$ let $(a_{j,1}^*, \dots, a_{j,r_j}^*)$ be a dual family to the family $(a_{j,1}, \dots, a_{j,r_j})$. For any fixed $j \in [d]$ and $i \in [r_j]$, applying the function $a_{j,i}^*$ to the original decomposition \begin{equation} \sum_{j=1}^d \sum_{i=1}^{r_j} a_{j,i}(x_j) b_{j,i}(\overline{x_j}) \label{original decomposition} \end{equation} of the zero tensor provides a decomposition of $b_{j,i}$ of the type \[b_{j,i}(\overline{x_j}) = \sum_{j' \neq j} \sum_{i'=1}^{r_{j'}} a_{j',i'}(x_{j'}) p_{((j,i),(j',i'))}(\overline{x_j, x_{j'}}).\] Substituting for every $j \in [d]$ and every $i \in [r_j]$ these expressions in the decomposition \eqref{original decomposition} we obtain \[\sum_{(j_1,j_2) \in \ct_2([d])} \sum_{i_1 \in [r_{j_1}], i_2 \in [r_{j_2}]} a_{j_1,i_1}(x_{j_1}) a_{j_2,i_2}(x_{j_2}) p_{((j_1,i_1),(j_2,i_2))}(\overline{x_{j_1}, x_{j_2}}).\] Gathering terms together whenever their sets $\{(j_1, i_1), (j_2, i_2)\}$ are the same, this decomposition of the zero tensor can be rearranged as \begin{equation} \sum_{1 \le j_1 < j_2 \le d} \sum_{i_1 \in [r_{j_1}], i_2 \in [r_{j_2}]} a_{j_1,i_1}(x_{j_1}) a_{j_2,i_2}(x_{j_2}) p_{\{(j_1,i_1),(j_2,i_2)\}}(\overline{x_{j_1}, x_{j_2}}). \label{decomposition after |J|=2} \end{equation} While doing this rearrangement we have potentially cancelled terms corresponding to $|J|=2$ in Definition \ref{zero form}.

After that, for any fixed $(j_1, j_2) \in \ct_2([d])$ and $(i_1,i_2) \in [r_{j_1}] \times [r_{j_2}]$, applying $a_{j_1,i_1}^* \otimes a_{j_2,i_2}^*$ leads to a decomposition \begin{align*} & p_{\{(j_1,i_1),(j_2,i_2)\}}(\overline{x_{j_1}, x_{j_2}}) = \sum_{j_3 \in [d] \setminus \{j_1, j_2\}} \sum_{i_3=1}^{r_{j_3}} a_{j_3,i_{j_3}}(x_{j_3}) p_{((j_1,i_1), (j_2,i_2), (i_3,j_3))}^{|J|=3, \mathrm{main}}(\overline{x_{j_1}, x_{j_2}, x_{j_3}}) + \\
& \sum_{(j_3, j_4) \in \ct_2([d] \setminus \{j_1, j_2\})} \sum_{i_3 \in [r_{j_3}], i_4 \in [r_{j_4}]} a_{j_3,i_{j_3}}(x_{j_3}) a_{j_4,i_{j_4}}(x_{j_4}) p_{((j_1,i_1), (j_2,i_2), (i_3,j_3), (i_4,j_4))}^{|J|=3, \mathrm{secondary}}(\overline{x_{j_1}, x_{j_2}, x_{j_3}, x_{j_4}}). \end{align*} The terms of the second type can in particular also be viewed as terms of the first type, so we may write the simpler expression \[p_{\{(j_1,i_1),(j_2,i_2)\}}(\overline{x_{j_1}, x_{j_2}}) = \sum_{j_3 \in [d] \setminus \{j_1, j_2\}} a_{j_3,i_{j_3}}(x_{j_3}) p_{((j_1,i_1), (j_2,i_2), (i_3,j_3))}(\overline{x_{j_1}, x_{j_2}, x_{j_3}}).\] Substituting in \eqref{decomposition after |J|=2} leads to the decomposition \[ \sum_{(j_1, j_2, j_3) \in \ct_3([d])} \sum_{i_1 \in [r_{j_1}], i_2 \in [r_{j_2}], i_3 \in [r_{j_3}]} a_{j_1,i_1}(x_{j_1}) a_{j_2,i_2}(x_{j_2}) a_{j_3,i_3}(x_{j_3}) p_{((j_1,i_1),(j_2,i_2),(j_3,i_3))}(\overline{x_{j_1}, x_{j_2}, x_{j_3}}). \] We then gather terms together, this time whenever their sets $\{(j_1, i_1), (j_2, i_2), (j_3,i_3)\}$ are the same. Our decomposition of the zero tensor now becomes \[\sum_{1 \le j_1 < j_2 < j_3 \le d} \sum_{i_1 \in [r_{j_1}], i_2 \in [r_{j_2}], i_3 \in [r_{j_3}]} a_{j_1,i_1}(x_{j_1}) a_{j_2,i_2}(x_{j_2}) a_{j_3,i_3}(x_{j_3}) p_{\{(j_1,i_1),(j_2,i_2),(j_3,i_3)\}}(\overline{x_{j_1}, x_{j_2}, x_{j_3}}).\] Again, while doing this rearrangement we may have cancelled terms corresponding to $|J|=3$ in Definition \ref{zero form}.

We then continue the process until our decomposition of the zero tensor has become of the type \[\sum_{1 \le j_1 < \dots < j_{d-1} \le d} \sum_{i_{j_1} \in [r_{j_1}], \dots, i_{j_{d-1}} \in [r_{j_{d-1}}]} \left( \prod_{1 \le t \le d-1} a_{j_t, i_{j_t}}(x_{j_t}) \right) p_{\{(j_1,i_1), \dots, (j_{d-1}, i_{d-1})\}}(x_j)\] where we write $j$ for the remaining element of $[d]$ aside from $j_1, \dots, j_{d-1}$. Until this point we may have cancelled terms corresponding to $2 \le |J| \le d-1$ in Definition \ref{zero form}. Applying products of the type \[a_{j_1,i_1}^* \otimes \dots \otimes a_{j_{d-1},i_{d-1}}^*\] we obtain that the functions \[p_{\{(j_1,i_1), \dots, (j_{d-1}, i_{d-1})\}}\] are linear combinations of the functions $a_{j,t_j}$. After cancelling terms corresponding to $|J|=d$ in Definition \ref{zero form}, our decomposition of the zero tensor becomes \[\sum_{i_1 \in [r_1], \dots, i_d \in [r_d]} \lambda_{i_1, \dots, i_d} a_{1,i_1} \otimes \dots \otimes a_{d,i_d} \] for some $\lambda_{i_1, \dots, i_d} \in \F$. These are necessarily all zero since the tensor products \[a_{1,i_1} \otimes \dots \otimes a_{d,i_d}\] are linearly independent. The result follows. \end{proof}

We finally move to the proof of Theorem \ref{At most d decompositions}. There again, we gradually split terms of decompositions as we did in the proof of Proposition \ref{order-$d$ decompositions of zero}, while carefully tracking down the specified one-variable functions in the terms that appear. The main guiding heuristic behind the computations that we will make is that at any point throughout the process, terms which have at least one specified one-variable function of the type $a_{j,i}^0$ as a factor are acceptable, as they do not get in the way of the desired conclusion, whereas for terms which have none of these we aim to ensure that all specified one-variable functions arising in the same product come from the same decomposition $\t$, and to that end we will use Theorem \ref{At most d decompositions} for tensors of lower-order, which leads to an inductive argument on the order $d$ of the tensor. As in Section \ref{Section: The proof for order-$3$ tensors}, notations such as $p_{(j,i),(j',i')}^{\t}$ will refer to functions depending only on a single decomposition $\t$, whereas notations such as $p_{(j,i),(j',i')}^{\t \t'}$ will refer to functions depending on both decompositions $\t, \t'$.

\begin{proof}[Proof of Theorem \ref{At most d decompositions}]

We proceed by induction on $d$. The base case is the case $d=2$, in other words Proposition \ref{A matrix with three common-independent decompositions has slice rank at most the common part}. Let us now assume $d \ge 3$, and suppose that the result has been shown for all $d' \in [2,d]$. We begin with the decompositions \begin{equation} \sum_{j=1}^d \left( \sum_{i=1}^{r_j^0} a_{j,i}^0(x_j) b_{j,i}^{0,\t}(\overline{x_j}) + \sum_{i=1}^{r_j^\t} a_{j,i}^{\t}(x_j) b_{j,i}^{\t}(\overline{x_j}) \right)\label{decompositions at the start of the proof of the sunflower result}\end{equation} of $T$ for each $\t \in [h]$ for some $h>d$. For $j \in [d]$ and every $\b \in \{0\} \cup [h]$ we define \[A_j^\b = \langle a_{j,i}^{\b}: i \in [r_j^\b] \rangle\] and for every $j \in [d]$ we choose a family \[(a_{j,i}^{\b*}: \b \in \{0\} \cup [h], i \in [r_j^\b])\] of dual functions to the family \[(a_{j,i}^{\b}: \b \in \{0\} \cup [h], i \in [r_j^\b]).\]

We start the first step of the splitting process. Let $j \in [d]$ and $i \in [r_j^1]$ be fixed. We write the equality between the decompositions $1$ and $2$ of $T$, and apply the function $a_{j,i}^{1*}$. This leads to an expression of the type \begin{align*} b_{j,i}^1(\overline{x_j}) & = \sum_{j' \neq j} \sum_{i'=1}^{r_{j'}^{1}} a_{j',i'}^1(x_{j'}) p_{(j,i),(j',i')}^{1}(\overline{x_j,x_{j'}}) \\
& - \sum_{j' \neq j} \sum_{i'=1}^{r_{j'}^{2}} a_{j',i'}^2(x_{j'}) p_{(j,i),(j',i')}^{12}(\overline{x_j,x_{j'}}) \\
& + \sum_{j' \neq j} \sum_{i'=1}^{r_{j'}^0} a_{j',i'}^0(x_{j'}) (p_{(j,i),(j',i')}^{0,1}-p_{(j,i),(j',i')}^{0,12})(\overline{x_j,x_{j'}}). \end{align*} We had started with at least $d+1$ decompositions of $T$, so there are at least $d$ of them that are distinct from $1$. Applying the inductive hypothesis to the order-$(d-1)$ tensor $b_{j,i}^1$, with for every $j' \neq j$ the functions $a_{j',i'}^0$ and $a_{j',i'}^1$ together forming a basis of center functions and the functions $a_{j',i'}^\t$ forming for each $\t \neq 1$ a basis of petal functions, we obtain \[b_{j,i}^1 \in (A_1^0 + A_1^1) \otimes (\otimes_{j' \notin \{j,1\}} \F^{n_{j'}}) + \dots + (\otimes_{j' \notin \{j,d\}} \F^{n_{j'}}) \otimes (A_d^0 + A_d^1).\] We likewise obtain such a description of the functions $b_{j,i}^{\t}$ for each $\t \in [h]$. Writing the equality between the decompositions $1$ and $2$ that we started with \eqref{decompositions at the start of the proof of the sunflower result} and applying the functions $a_{j,i}^{0*}$ we obtain \[b_{j,i}^{0,1} - b_{j,i}^{0,2} \in (A_1^0 + A_1^1 + A_1^2) \otimes (\otimes_{j' \notin \{j,1\}} \F^{n_{j'}}) + \dots + (\otimes_{j' \notin \{j,d\}} \F^{n_{j'}}) \otimes (A_d^0 + A_d^1 + A_d^2).\]

We write $S_0$ for the part \[\sum_{j=1}^d \sum_{i=1}^{r_j^0} a_{j,i}^0(x_j) b_{j,i}^{0,1}(\overline{x_j})\] of the decomposition $1$ of $T$ with central one-variable functions. This tensor belongs to the subspace \[R_{\mathrm{target}}= (A_1^0 \otimes \F^{n_2} \otimes \dots \otimes \F^{n_d}) + \dots + (\F^{n_1} \otimes \dots \otimes \F^{n_{d-1}} \otimes A_d^0).\]

We then consider decompositions of $T-S_0$. The equality between the resulting decompositions $1$ and $2$ can be written as \begin{equation} T_1 + T_{\Delta} = T_2, \label{equality between T_1, T_2, T_{Delta} in the general case} \end{equation} with 
\begin{align} T_1 & \in +_{(j_1,j_2) \in \ct_2([d])} (A_{j_1}^1 \otimes (A_{j_2}^0 \oplus A_{j_2}^1) \otimes (\otimes_{j \notin \{j_1,j_2\}} \F^{n_{j}})) \label{decomposition of T_1 in the proof for general d}\\
T_2 & \in +_{(j_1,j_2) \in \ct_2([d])} (A_{j_1}^2 \otimes (A_{j_2}^0 \oplus A_{j_2}^2) \otimes (\otimes_{j \notin \{j_1,j_2\}} \F^{n_{j}})) \label{decomposition of T_2 in the proof for general d}\\
T_\Delta & \in +_{(j_1,j_2) \in \ct_2([d])} (A_{j_1}^0 \otimes (A_{j_2}^0 \oplus A_{j_2}^1 \oplus A_{j_2}^2) \otimes (\otimes_{j \notin \{j_1,j_2\}} \F^{n_{j}})) \label{decomposition of T_Delta in the proof for general d} \end{align} where the unspecified $(d-2)$-variable functions in $\otimes_{j \notin \{j_1,j_2\}} \F^{n_{j}}$ from \eqref{decomposition of T_1 in the proof for general d} and \eqref{decomposition of T_2 in the proof for general d} respectively depend on the decompositions $1$ only and $2$ only.

In any of these three tensors, any term of the decomposition that we have obtained is split as a product of two specified one-variable functions and one unspecified $(d-2)$-variable function. Whenever some of the resulting terms have the same pair of specified one-variable functions, we gather them together into a single term, even if the original two terms come from different sides of the equality. For instance, the two terms \begin{align*} & a_{1,i_1}^{1}(x_1) a_{2,i_2}^{0}(x_2) p' (x_3, \dots, x_d) \\ & a_{1,i_1}^{1}(x_1) a_{2,i_2}^{0}(x_2) p''(x_3, \dots, x_d) \end{align*} for some unspecified functions $p', p''$ would be replaced by a single term \[a_{1,i_1}^{1}(x_1) a_{2,i_2}^{0}(x_2) p(x_3, \dots, x_d).\] Because we gather terms even if they come from different sides of the original equality, the resulting equality is in general no longer between two decompositions of $T-S_0$, but between two decompositions of $T-S_1$ for some new tensor $S_1$. However, since in the expressions of $T_1, T_2, T_{\Delta}$ given by \eqref{decomposition of T_1 in the proof for general d}, \eqref{decomposition of T_2 in the proof for general d}, \eqref{decomposition of T_Delta in the proof for general d} there is no term with both a function of the type $a_{j,i}^1$ and a function of the type $a_{j,i}^2$ as its specified two one-variable functions, any term which has neither of the two specified one-variable functions of the type $a_{j,i}^0$ must necessarily have both its specified one-variable functions of the type $a_{j,i}^1$, or both its specified one-variable functions of the type $a_{j,i}^2$, in which case the term may only arise in the left-hand side or in the right-hand side of \eqref{equality between T_1, T_2, T_{Delta} in the general case} respectively. The difference $S_1-S_0$ hence only involves terms which have some $a_{j,i}^0$ as a factor; in other words $S_1-S_0$ belongs to $R_{\mathrm{target}}$. Since $S_0$ belongs to $R_{\mathrm{target}}$, it follows that $S_1$ belongs to $R_{\mathrm{target}}$. Without loss of generality this equality is of the type \begin{align} & \sum_{1 \le j_1 < j_2 \le d} \sum_{i_1 \in [r_{j_1}^1], i_2 \in [r_{j_2}^1]} a_{j_1,i_1}^1(x_{j_1}) a_{j_2,i_2}^1(x_{j_2}) p_{\{(j_1,i_1),(j_2,i_2)\}}^1(\overline{x_{j_1},x_{j_2}}) \nonumber \\
+ & \sum_{(j_1,j_2) \in \ct_2([d])} \sum_{\b \in \{0,1,2\}} \sum_{i_1 \in [r_{j_1}^0], i_2 \in [r_{j_2}^\b]} a_{j_1,i_1}^0(x_{j_1}) a_{j_2,i_2}^\b(x_{j_2}) p_{((j_1,i_1),(j_2,i_2)), \b}^{12}(\overline{x_{j_1},x_{j_2}}) \nonumber \\
= & \sum_{1 \le j_1 < j_2 \le d} \sum_{i_1 \in [r_{j_1}^2], i_2 \in [r_{j_2}^2]} a_{j_1,i_1}^2(x_{j_1}) a_{j_2,i_2}^2(x_{j_2}) p_{\{(j_1,i_1),(j_2,i_2)\}}^2(\overline{x_{j_1},x_{j_2}}) \label{equality after the first step} \end{align} since terms with a factor of the type $a_{j_1,i_1}^0$ may be collected on either side of the equality. This finishes the first step of the splitting process.

We now start the second step. In the equality \eqref{equality after the first step}, we consider some distinct elements $j_1,j_2$ of $[d]$ and some $i_1 \in [r_{j_1}^1]$, $i_2 \in [r_{j_2}^1]$, and apply $a_{j_1,i_1}^{1*} \otimes a_{j_2,i_2}^{1*}$. This leads to \[p_{\{(j_1,i_1),(j_2,i_2)\}}^1 \in +_{j \notin \{j_1,j_2\}} ((A_{j}^0 \oplus A_{j}^1 \oplus A_{j}^2) \otimes (\otimes_{j' \notin \{j_1,j_2,j\}} \F^{n_{j'}})).\] Applying Theorem \ref{At most d decompositions} for order-$(d-2)$ tensors, with for every $j \notin \{j_1,j_2\}$ the functions $a_{j,i}^0$ and $a_{j,i}^1$ together forming a basis of center functions and the functions $a_{j,i}^\t$ forming for each $\t \neq 1$ a basis of petal functions, we obtain \[p_{\{(j_1,i_1),(j_2,i_2)\}}^1 \in +_{j \notin \{j_1,j_2\}} ((A_{j}^0 \oplus A_{j}^1) \otimes (\otimes_{j' \notin \{j_1,j_2,j\}} \F^{n_{j'}}))\] where the unspecified $(d-3)$-variable functions in $\otimes_{j' \notin \{j_1,j_2,j\}} \F^{n_{j'}}$ depend on the decomposition $1$ only. We obtain a similar statement for the function $p_{\{(j_1,i_1),(j_2,i_2)\}}^2$. Returning to the equality \eqref{equality after the first step} between two decompositions of $T-S_1$, we apply the other products, that is, those of any of the types \[a_{j_1,i_1}^{0*} \otimes a_{j_2,i_2}^{0*}\text{, }a_{j_1,i_1}^{0*} \otimes a_{j_2,i_2}^{1*}\text{, } a_{j_1,i_1}^{0*} \otimes a_{j_2,i_2}^{2*}\] (as there are no terms where a product of the type $a_{j_1,i_1}^{1} \otimes a_{j_2,i_2}^{2}$ appears, there is no need to apply any $a_{j_1,i_1}^{1*} \otimes a_{j_2,i_2}^{2*})$. We obtain \[p_{((j_1,i_1),(j_2,i_2)), \b}^{12} \in +_{j \notin \{j_1,j_2\}} ((A_{j}^0 \oplus A_{j}^1 \oplus A_{j}^2) \otimes (\otimes_{j' \notin \{j_1,j_2,j\}} \F^{n_{j'}}))\] for all $(j_1,j_2) \in \ct_2([d])$, all $\b \in \{0,1,2\}$ and all $i_1 \in [r_{j_1}^0]$ and $i_2 \in [r_{j_2}^\b]$.

We then substitute in the equality \eqref{equality after the first step} between decompositions of $T-S_1$, and are left with an equality between two sides which are each sums of products of four functions, with each product containing three specified one-variable functions and one unspecified function in the remaining $d-3$ variables. We factor together any terms which have the same three specified one-variable functions, even if this requires changing which side of the equality they are on. This leads to an equality between two decompositions of $T-S_2$ for some new tensor $S_2$. We note that the triples of specified one-variable functions are of one of two types: either they contain at least one term of the type $a_{j,i}^0$, or they do not, in which case they are necessarily of one of the types \begin{equation} a_{j_1,i_1}^1 \otimes a_{j_2,i_2}^1 \otimes a_{j_3,i_3}^1 \text{, } a_{j_1,i_1}^2 \otimes a_{j_2,i_2}^2 \otimes a_{j_3,i_3}^2. \label{product of three functions of the same type} \end{equation} Furthermore, if a term has its product of specified one-variable functions of these two latter types, then it may only arise in the left-hand side or in the right-hand side respectively. All terms involved in the difference $S_2-S_1$ hence have a function of the type $a_{j,i}^0$ as a factor, so $S_2-S_1$ belongs to $R_{\mathrm{target}}$. Since $S_1$ also does, we conclude that $S_2$ belongs to $R_{\mathrm{target}}$.

We emphasize that we have shown that for terms which have their product of three specified one-variable functions as the first (resp. the second) of \eqref{product of three functions of the same type}, the unspecified $(d-3)$-variable function depends only on the decomposition 1 (resp. 2). This guarantees that the next time that we apply Theorem \ref{At most d decompositions} to a lower-order tensor, we indeed apply it to several decompositions of the same tensor. (Likewise, when applying Theorem \ref{At most d decompositions} in the step that we are now completing, we have used that the function $p_{\{(j_1,i_1),(j_2,i_2)\}}^1$ was the same whenever we wrote an equality between decompositions from $\t=1$ and from any other value of $\t$.) We have now obtained an equality of the type \begin{align} & \sum_{1 \le j_1 < j_2 < j_3 \le d} \sum_{i_1 \in [r_{j_1}^1], i_2 \in [r_{j_2}^1], i_3 \in [r_{j_3}^1]} a_{j_1,i_1}^1(x_{j_1}) a_{j_2,i_2}^1(x_{j_2}) a_{j_3,i_3}^1(x_{j_3}) p_{\{(j_1,i_1),(j_2,i_2), (j_3,i_3)\}}^1(\overline{x_{j_1},x_{j_2},x_{j_3}}) \nonumber \\
+ & \sum_{1 \le j_1 < j_2 < j_3 \le d} \sum_{(\b_1, \b_2, \b_3) \in \{0,1,2\}^3 \setminus \{1,2\}^3} \sum_{i_1 \in [r_{j_1}^{\b_1}], i_2 \in [r_{j_2}^{\b_2}], i_3 \in [r_{j_3}^{\b_3}]} a_{j_1,i_1}^{\b_1}(x_{j_1}) a_{j_2,i_2}^{\b_2}(x_{j_2}) a_{j_3,i_3}^{\b_3}(x_{j_3}) \nonumber \\
&p_{\{(j_1,i_1, \b_1),(j_2,i_2, \b_2),(j_3,i_3, \b_3)\}}^{12}(\overline{x_{j_1},x_{j_2},x_{j_3}}) \nonumber \\
= & \sum_{1 \le j_1 < j_2 < j_3 \le d} \sum_{i_1 \in [r_{j_1}^2], i_2 \in [r_{j_2}^2], i_3 \in [r_{j_3}^2]} a_{j_1,i_1}^2(x_{j_1}) a_{j_2,i_2}^2(x_{j_2}) a_{j_3,i_3}^2(x_{j_3}) p_{\{(j_1,i_1),(j_2,i_2), (j_3,i_3)\}}^2(\overline{x_{j_1},x_{j_2},x_{j_3}}). \nonumber \end{align}

This finishes the second step of the splitting process. We then continue the process in a similar way. For every $m \in [d-1]$, at the end of the $m-1$th step we have obtained a tensor $S_{m-1}$ in $R_{\mathrm{target}}$, and an equality between two decompositions of $T - S_{m-1}$ of the type \begin{align} & \sum_{1 \le j_1 < \dots < j_m \le d} \sum_{i_1 \in [r_{j_1}^1], \dots, i_m \in [r_{j_m}^1]} a_{j_1,i_1}^1(x_{j_1}) \dots a_{j_m,i_m}^1(x_{j_m}) p_{\{(j_1,i_1),\dots,(j_m,i_m)\}}^1(\overline{x_{j_1},\dots,x_{j_m}}) \nonumber \\
+ & \sum_{1 \le j_1 < \dots < j_m \le d} \sum_{(b_1, \dots, \b_m) \in \{0,1,2\}^m \setminus \{1,2\}^m} \sum_{i_1 \in [r_{j_1}^{\b_1}], \dots, i_m \in [r_{j_m}^{\b_m}]} a_{j_1,i_1}^{\b_1}(x_{j_1}) \dots a_{j_m,i_m}^{\b_m}(x_{j_m}) \nonumber \\
&p_{\{(j_1,i_1,\b_1),\dots,(j_m,i_m,\b_m)\}}^{12}(\overline{x_{j_1},\dots,x_{j_m}}) \nonumber \\
= & \sum_{1 \le j_1 < \dots < j_m \le d} \sum_{i_1 \in [r_{j_1}^2], \dots, i_m \in [r_{j_m}^2]} a_{j_1,i_1}^2(x_{j_1}) \dots a_{j_m,i_m}^2(x_{j_m}) p_{\{(j_1,i_1),\dots,(j_m,i_m)\}}^2(\overline{x_{j_1},\dots,x_{j_m}}). \nonumber \end{align} The important features of this equality are the following.

\begin{enumerate}[(i)]

\item All terms of both sides are products of $m+1$ functions, $m$ of which are specified one-variable functions of the type $a_{j,i}^\b$ and the remaining one of which is an unspecified function in the other $d-m$ variables.
\item Whenever $j_1,\dots,j_m$ are pairwise distinct indices in $[d]$, $\b_1, \dots, \b_m$ are in $\{0,1,2\}$, and $i_1 \in [r_{j_1}^{\b_1}], \dots, i_m \in [r_{j_m}^{\b_m}]$ are indices, there is only at most one term in total (from both the left-hand and right-hand sides) that has \[\otimes_{1 \le t \le m} a_{j_t,i_t}^{\b_t}\] as the product of its $m$ specified one-variable functions.
\item Whenever a term does not have any function of the type $a_{j,i}^0$ as a factor, its product of specified one-variable functions is of one of the types \[\otimes_{1 \le t \le m} a_{j_t,i_t}^{1} \text{, } \otimes_{1 \le t \le m} a_{j_t,i_t}^{2},\] and it must come from the left-hand side and from the right-hand side respectively in these cases.
\item In the situation of terms described in the previous item, the remaining unspecified $(d-m)$-variable functions depend only on the decompositions $1$ and $2$ respectively.

\end{enumerate}

In particular, after $d-2$ steps we have an equality between two decompositions of $T - S_{d-2}$, that obeys the properties above for $m=d-1$. This equality is of the type \begin{align*} & \sum_{j=1}^d \sum_{i_{[d]\setminus \{j\}} \in \prod_{j' \neq j} [r_{j'}^1]} \left( \prod_{j' \neq j} a_{j', i_{j'}}^1(x_{j'}) \right) p_{j,i_{[d] \setminus \{j\}}}^1(x_j)\\
+ & \sum_{j=1}^d \sum_{\b_{[d]\setminus \{j\}} \in \{0,1,2\}^{[d] \setminus \{j\}} \setminus \{1,2\}^{[d] \setminus \{j\}}} \sum_{i_{[d]\setminus \{j\}} \in \prod_{j' \neq j} [r_{j'}^{\b_{j'}}]} \left( \prod_{j' \neq j} a_{j', i_{j'}}^{\b_{j'}}(x_{j'}) \right) p_{j,i_{[d] \setminus \{j\}}, \b_{[d] \setminus \{j\}}}^{12}(x_j)\\
= & \sum_{j=1}^d \sum_{i_{[d]\setminus \{j\}} \in \prod_{j' \neq j} [r_{j'}^2]} \left( \prod_{j' \neq j} a_{j', i_{j'}}^2(x_{j'}) \right) p_{j,i_{[d] \setminus \{j\}}}^2(x_j) \end{align*} where we have written $i_{[d]\setminus \{j\}}$ for $(i_{j'})_{j' \in [d] \setminus \{j\}}$ and $\b_{[d]\setminus \{j\}}$ for $(\b_{j'})_{j' \in [d] \setminus \{j\}}$.

Each term is now a product of $d$ one-variable functions, where $d-1$ of them are specified functions of the type $a_{j,i}^{\b}$ and the last one is an unspecified function. Applying one last time the products of the $d-1$ one-variable specified functions shows that the unspecified function is a linear combination of functions of the type $a_{j,i}^{\b}$ (for the relevant value of $j$). After yet again gathering terms, we obtain some tensor $S_{d-1} \in R_{\mathrm{target}}$ and an equality between two decompositions of $T - S_{d-1}$ of the type \begin{align} & \sum_{i_1 \in [r_1^1], \dots, i_d \in [r_d^1]} \l_{(1,i_1,1), \dots, (d,i_d,1)} a_{1,i_1}^1(x_1) \dots a_{d,i_d}^1(x_d) \nonumber \\
+ & \sum_{j=1}^d \sum_{i_{[d]\setminus \{j\}} \in \prod_{j' \neq j} [r_{j'}^1]} \l_{(1,i_1,1), \dots, (j-1, i_{j-1},1), (j, i_{j},2), (j+1, i_{j+1},1), \dots, (d,i_d,1)} \left( \prod_{j' \neq j} a_{j', i_{j'}}^1(x_{j'}) \right) a_{j,i_j}^2(x_j)\nonumber \\
+ & \sum_{(\b_1, \dots, \b_d) \in \{0,1,2\}^d \setminus \{1,2\}^d} \sum_{i_1 \in [r_1^{\b_1}], \dots, i_d \in [r_d^{\b_d}]} \l_{(1,i_1,\b_1), \dots, (d,i_d,\b_d)} a_{1,i_1}^{\b_1}(x_1) \dots a_{d,i_d}^{\b_d}(x_d)\nonumber \\
= & \sum_{j=1}^d \sum_{i_{[d]\setminus \{j\}} \in \prod_{j' \neq j} [r_{j'}^2]} \l_{(1,i_1,2), \dots, (j-1, i_{j-1},2), (j, i_{j},1), (j+1, i_{j+1},2), \dots, (d,i_d,2)} \left( \prod_{j' \neq j} a_{j', i_{j'}}^2(x_{j'}) \right) a_{j,i_j}^1(x_j)\nonumber \\
+ & \sum_{i_1 \in [r_1^2], \dots, i_d \in [r_d^2]} \l_{(1,i_1,2), \dots, (d,i_d,2)} a_{1,i_1}^2(x_1) \dots a_{d,i_d}^2(x_d) \label{equality at the end of the argument} \end{align} 
for some $\l_{(1,i_1,\b_1), \dots, (d,i_d,\b_d)} \in \F$. In this equality, all terms on both sides are multiples of products of the type \begin{equation} a_{1,i_1}^{\b_1} \otimes \dots \otimes a_{d,i_d}^{\b_d}. \label{products with j,i,beta} \end{equation} If a term has $\b_j \neq 0$ for every $j \in [d]$, then it has either at least $d-1$ or at most one index equal to $1$ (and the remaining ones are equal to $2$), and in these two cases it necessarily must come from the left-hand and right-hand side respectively, so the sets of values of $(\b_1, \dots, \b_d)$ from the terms appearing in both sides are disjoint. Because the products \eqref{products with j,i,beta} are linearly independent, it follows that either side of \eqref{equality at the end of the argument} is equal to $0$. These sides are equal to $T-S_{d-1}$, so we have shown $T=S_{d-1}$. Therefore, $T$ belongs to $R_{\mathrm{target}}$. In particular it has the desired decomposition and has slice rank at most \[r_1^0 + \dots + r_d^0.\qedhere\]

\end{proof}

\section{Open questions}\label{Section: Open questions}

Let us finish by discussing a few questions that are left open by our current results and proofs. As we had explained in Section \ref{Section: the tensor rank case}, in the tensor rank case the exponent in the upper bound that we show on the number of minimal-length tensor rank decompositions cannot be improved by more than a factor of $d$, but this nonetheless does not give a complete answer to our first question.

\begin{question}

What are the optimal bounds in Proposition \ref{tensor rank case}, and for which tensors are they attained ?

\end{question}

Returning to the slice rank, Example \ref{matrix times slice rank 2r example} shows that Theorem \ref{Main theorem} is false for infinite fields, that in the finite field case the bounds cannot be uniform with respect to $\F$, and furthermore that even in the finite field case they grow square-exponentially in the slice rank of the tensor. This leads us instead to the following formulation, which we believe provides a statement of a similar kind and which is uniform over all fields.

\begin{conjecture} \label{New formulation} Let $d \ge 3$, $k \ge 1$ be integers. Then there exists a positive integer $C(d,k)$ such that the following holds. Let $T$ be an order-$d$ tensor over some arbitrary field $\F$ and with slice rank equal to $k$. Then there exist linear subspaces \[A_1 \subset \F^{n_1}, \dots, A_d \subset \F^{n_d}\] each with dimension at most $C(d,k)$ such that if $r_1, \dots, r_d$ are nonnegative integers satisfying \[r_1+\dots+r_d=k\] and $a_{j,i}: [n_j] \to \F$ with $j \in [d]$ and $i \in [r_j]$ are functions satisfying \[T(x_1, \dots, x_d) = \sum_{j=1}^d \sum_{i=1}^{r_j} a_{j,i}(x_j) b_{j,i}(\overline{x_j})\] for some $(d-1)$-variable functions $b_{j,i}: \prod_{j' \neq j} [n_{j'}] \to \F$, then \[\langle a_{1,1},\dots,a_{1,r_1} \rangle \subset A_1, \dots, \langle a_{d,1},\dots,a_{d,r_d} \rangle \subset A_d.\] \end{conjecture}

We would furthermore not be surprised if for any fixed $d \ge 3$ we could take $C(k,d)$ to be linear in $k$, as we have not managed to build a counterexample disproving this.

\begin{question}

Is Conjecture \ref{New formulation} true with furthermore $C(k,d) = O(k)$ for every $d \ge 3$ ?

\end{question}

One further direction in which we believe that our results can be taken further is that of formulating and proving suitable generalisations of Theorem \ref{At most d decompositions} and Theorem \ref{Main theorem} for other notions of rank. As described in the introduction, one such notion that is of interest is the partition rank.

\begin{question}

Can we prove an analogue of Theorem \ref{Main theorem} for the partition rank ?

\end{question}

As was the case for the slice rank, formulating the property already appears to be challenging, and it is useful to keep in mind that the five ways in which various slice rank decompositions can arise all extend to the partition rank. Let us finish by describing these extensions. For $d \ge 2$ an integer, we write $\cb_2([d])$ for the set of bipartitions $\{J,J^c\}$ of $[d]$ with $J,J^c \neq \emptyset$. We recall that if $d \ge 2$ is an integer, $J$ is a subset of $[d]$, then for every $x \in [n_1] \times \dots \times [n_d]$ we write $x(J)$ for the restriction of $x$ to its coordinates in $J$. Let us also recall the definition of the partition rank.

\begin{definition} Let $d \ge 2$ be an integer, and let $T: [n_1] \times \dots \times [n_d] \to \F$ be an order-$d$ tensor. We say that the \emph{partition rank} of $T$, denoted by $\pr T$, is the smallest nonnegative integer $k$ such that there exist nonnegative integers $r_{\{J,J^c\}}$ for each bipartition $\{J,J^c\} \in \cb_2([d])$, such that \[\sum_{\{J,J^c\} \in \cb_2([d])} r_{\{J,J^c\}} = k\] and satisfying one of the following two equivalent properties. \begin{enumerate}

\item There exist matrices $M_{\{J,J^c\}}: (\prod_{j \in J} [n_{j}]) \times (\prod_{j \in J^c} [n_{j}]) \to \F$ with rank at most $r_{\{J,J^c\}}$ for each $\{J,J^c\} \in \cb_2([d])$ such that \begin{equation} T(x_1, \dots, x_d) = \sum_{\{J,J^c\} \in \cb_2([d])} M_{\{J,J^c\}}(x(J), x(J^c)) \label{partition rank decomposition with matrices} \end{equation} is satisfied for all $(x_1, \dots, x_d) \in [n_1] \times \dots \times [n_d]$.

\item For each $\{J,J^c\} \in \cb_2([d])$ and each $i \in [r_{\{J,J^c\}}]$ there exist some functions $a_{J,i}: \prod_{j \in J} [n_{j}] \to \F$ and $b_{J^c,i}: \prod_{j \in J^c} [n_{j}] \to \F$ such that \begin{equation} T(x_1, \dots, x_d) = \sum_{\{J,J^c\} \in \cb_2([d])} \sum_{i=1}^{r_{\{J,J^c\}}} a_{J,i}(x(J)) b_{J^c,i}(x(J^c)) \label{partition rank decomposition} \end{equation} is satisfied for all $(x_1, \dots, x_d) \in [n_1] \times \dots \times [n_d]$. \end {enumerate}

We say that an expression such as \eqref{partition rank decomposition} is a \emph{partition rank decomposition} of $T$, say that the \emph{length} of the decomposition is the integer \[\sum_{\{J,J^c\} \in \cb_2([d])} r_{\{J,J^c\}},\] and say that the decomposition has \emph{minimal length} if its length is equal to the partition rank of $T$. \end{definition}

Constructions \ref{Construction 1} to \ref{Construction 5} respectively adapt to the following. Let $d \ge 3$, $k \ge 1$ be integers, and let $T$ be an order-$d$ tensor with partition rank $k$. 

\begin{construction} \emph{If $M_{\{J,J^c\}}: (\prod_{j \in J} [n_j]) \times (\prod_{j \in J^c} [n_j]) \to \F$ with $\{J,J^c\} \in \cb_2([d])$ are matrices satisfying \eqref{partition rank decomposition with matrices} and \[\sum_{\{J,J^c\} \in \cb_2([d])} \rk M_{\{J,J^c\}} = \pr T,\] then rewriting any minimal-length matrix decomposition of any of the matrices $M_{\{J,J^c\}}$ leads to a new minimal-length partition rank decomposition of $T$.} \end{construction}

\begin{construction} \emph{If the tensor rank of $T$ is equal to the partition rank $k$ of $T$, and \[T(x_1, \dots, x_d)  = \sum_{i=1}^k a_{1,i}(x_1) \dots a_{d,i}(x_d)\] is a length-$k$ tensor rank decomposition of $T$ then for any map $\iota: [k] \to \cb_2([d])$ the decomposition \[T(x_1, \dots, x_d) = \sum_{\{J,J^c\} \in \cb_2([d])} \sum_{i \in \iota^{-1}(\{J,J^c\})} (\prod_{j \in J} a_{j,i})(x(J)) (\prod_{j \in J^c} a_{j,i})(x(J^c))\] is a minimal-length partition rank decomposition of $T$. This is merely the most extreme form that this construction can take, and tensors $T$ satisfying much weaker assumptions than $\pr T = \tr T$ will lend themselves to it. For instance if $T$ has partition rank $k$ and can merely be written as \[T(x_1, \dots, x_d) = \sum_{i=1}^k a_{i,J_{i,1}}(x(J_{i,1})) a_{i,J_{i,2}}(x(J_{i,2})) a_{i,J_{i,3}}(x(J_{i,3})) \] where for every $i \in [k]$ the set $\{J_{i,1},J_{i,2},J_{i,3}\}$ is a tripartition of $[d]$ with $J_{i,1}, J_{i,2}, J_{i,3}$ each non-empty and $a_{i,J_{i,s}}:\prod_{j \in J_{i,s}} [n_{j}] \to \F$ is a function for each $s=1,2,3$ then we already obtain different minimal-length decompositions of $T$ depending on how we view each summand as a product of two terms.}\end{construction}

\begin{construction} \emph{A tensor $[k]^d \to \F$ usually has partition rank equal to $k$. Indeed, for any $d \ge 3$, if the field $\F$ is finite then there are at most \[k^{2^d} |\F|^{(k-1)(k^{d-1}+k)} \le (k^{2^d} / |\F|^{k}) |\F|^{k^d}\] tensors with partition rank at most $k-1$. So writing $T$ as the sum of its $(d-1)$-variable slices as in Construction \ref{Construction 3} again provides $d$ minimal-length partition rank decomposition of $T$. We can also write $T$ as a sum of its $d'$-variable slices for $d'<d-1$, but the resulting partition rank decomposition then has length $k^{d-d'} > k$, which is not minimal, and we hence have less of a generalisation in this direction.} \end{construction}

\begin{construction} \label{Construction 9} \emph{If the decomposition \[T(x_1, \dots, x_d)  = \sum_{\{J,J^c\} \in \cb_2([d])} \sum_{i=1}^{r_{\{J,J^c\}}} a_{J,i}(x(J)) a_{J^c,i}(x(J^c))\] is a minimal-length decomposition of $T$, $\{J_1,J_2,J\}$ is a tripartition of $[d]$ with $J_1, J_2$ each with size at least $1$, $i_1 \in [r_{\{J_1, J_1^c\}}]$ and $i_2 \in [r_{\{J_2, J_2^c\}}]$ are indices, and $c: \prod_{j \in J} [n_j] \to \F$ is a function if $J$ is non-empty and an element of $\F$ otherwise, then replacing $a_{J_1^c,i_1}$ and $a_{J_2^c,i_2}$ by respectively \[a_{J_1^c,i_1} + a_{I_2,i_2} \otimes c \text{ and } a_{J_2^c,i_2} - a_{J_1,i_1} \otimes c\] leads to a new minimal-length partition rank decomposition of $T$.} \end{construction}

\begin{construction} \emph{More generally, if $s \in [2,d]$ is an integer, $\{J_1, \dots, J_s, J\}$ is a partition of $[d]$ with $J_1, \dots, J_s$ each with size at least $1$, $i_1 \in [r_{\{J_1, J_1^c\}}], \dots, i_s \in [r_{\{J_s, J_s^c\}}]$ are indices, and $c_1, \dots, c_s$ are functions $\prod_{j \in J} [n_j] \to \F$ if $J$ is non-empty and elements of $\F$ otherwise, which satisfy \[c_1 + \dots + c_s = 0,\] then replacing $a_{J_t^c,i_t}$ by \[a_{J_t^c,i_t} + (\otimes_{t’ \in [s] \setminus \{t\}} a_{J_{t'},i_{t'}}) \otimes c_t\] for each $t \in [s]$ leads to a new minimal-length partition rank decomposition of $T$.} \end{construction}

Construction \ref{Construction 9} shows that whenever $|J| \ge 2$, it is never true (for any integers $d \ge 3$, $k \ge 2$, and any field $\F$) that there exists a linear subspace $A_J$ of $\otimes_{j \in J} \F^{n_j}$ with dimension bounded above depending on $d,k,\F$ only such that any functions $a_{J,i}$ from any minimal-length decomposition of $T$ are contained in $A_J$.

Moreover, in the case of the partition rank there may be new constructions in addition to the analogues above.

\end{document}